\documentclass[12pt]{article}

\usepackage{amsthm}

\usepackage{amssymb,amsmath}
\usepackage{graphics,epsfig,longtable,setspace,color,xcolor}
\usepackage{mathtools}
\usepackage{lipsum}
\usepackage{amsfonts}
\usepackage{graphicx}
\usepackage{epstopdf}
\usepackage{algorithmic}
\usepackage{amsopn}
\usepackage{mhchem}
\usepackage{longtable}

\newtheorem{theorem}{Theorem}[section]
\newtheorem{lemma}[theorem]{Lemma}
\newtheorem{remark}[theorem]{Remark}
\newtheorem{definition}[theorem]{Definition}
\newtheorem{example}[theorem]{Example}

\newtheorem{proposition}[theorem]{Proposition}
\newtheorem{corollary}[theorem]{Corollary}
\newtheorem{assumptions}[theorem]{Assumptions}

\numberwithin{equation}{section}

\def\R{\mathbb R}
\def\S{\mathbb S}
\def\div{{\rm div}\,}
\newcommand{\1}{\mathbf 1}
\def\a{{\bf a}}

\def\z{{\bf z}}
\def\v{{\bf v}}

\def\LL{\mathcal{L}}
\renewcommand{\H}{{\mathcal H}}

\begin{document}

\title{Anisotropic tempered diffusion equations}
\author{Juan Calvo, Antonio Marigonda and Giandomenico Orlandi}

\date{\today}
\maketitle

\begin{abstract}
We introduce a functional framework which is specially suited to formulate several classes of anisotropic evolution equations of tempered diffusion type. Under an amenable set of hypothesis involving a very natural potential function, these models can be shown to belong to  the entropy solution framework devised by \cite{Andreu2005elliptic,Andreu2005parabolic}, therefore ensuring well-posedness. We connect the properties of this potential with those of the associated cost function, thus providing a link with optimal transport theory and a supply of new examples of relativistic cost functions. Moreover, we characterize the anisotropic spreading properties of these models and we determine the Rankine--Hugoniot conditions that rule the temporal evolution of jump hypersurfaces under the given anisotropic flows.
\end{abstract}

\section{Introduction}
 \label{sec:intro}
 The goal of this document is to analyze the properties of a class of tempered diffusion equations having the form
 \begin{equation}
\label{eq:anisotropic_template0}
\frac{\partial u}{\partial t} = \div \left(u\,  \psi \left(\nabla u/u\right)\right)
\end{equation}
 where $\psi$ satisfies certain assumptions enforcing anisotropic spreading behavior. This work is a continuation of \cite{Calvo2015}, where the isotropic situation was analyzed. Here we reformulate that framework in a more synthetic way that enables us to incorporate anisotropies  naturally.
 
 Tempered diffusion equations are a class of degenerate parabolic equations in divergence form, characterized by the fact that their flux saturates to a constant value whenever the size of solution's gradient is large enough. This class comes also under the names of ``flux-saturated'' or `flux-limited'' diffusion equations in the mathematical literature. Actually, the most studied example of a tempered diffusion equation in the mathematical literature is the so-called ``relativistic heat equation'':
\begin{equation}
\label{eq:rhe0}
\frac{\partial u}{\partial t} = \nu \, \div \left( \frac{u \nabla u}{\sqrt{u^2 +\frac{\nu^2}{c^2} |\nabla u|^2}}\right)\:.
\end{equation}
This model was first introduced in \cite{Rosenau92}; the ``relativistic heat equation'' terminology was coined in \cite{Brenier1}, where important connections with optimal transport theory were pointed out. The properties of \eqref{eq:rhe0} have been studied both from analytical and numerical points of view in a long series of papers \cite{Andreu2006,Andreu2008,Andreu2010,Andreu2012,Calvo2017,Carrillo2013,Marquina,Serna} -many variants of this model have been also considered, see e.g. \cite{Calvo2013,Calvo2015survey,Calvo2016,Campos} and references therein. As a matter of fact, this nonlinear diffusion equation combines well-known properties of parabolic and hyperbolic equations. Our present knowledge about those properties is far from being complete, though. We can get some intuition about this mixed behavior by considering the following asymptotic regimes: when $c\to \infty$ we obtain the standard heat equation (see \cite{Caselles2007})
\begin{equation}
\label{eq:he}
\frac{\partial u}{\partial t} = \nu \Delta u,
\end{equation}
whereas for $\nu \to \infty$ we arrive to (see \cite{Andreu2007})
\begin{equation}
\label{eq:trmedia}
\frac{\partial u}{\partial t} = c\, \div \left(u \frac{Du}{|Du|} \right).
\end{equation}
Note that \eqref{eq:trmedia} is a hyperbolic equation with finite propagation speed given by $c$ which is capable of supporting discontinuous fronts. Those properties of (\ref{eq:trmedia}) are inherited by (\ref{eq:rhe0}) \cite{Andreu2006}; actually, this is due to the fact that the behavior of solutions to  (\ref{eq:rhe0}) at their interfaces is determined essentially by (\ref{eq:trmedia}). On the other hand, the behavior in the bulk is controlled by (\ref{eq:he}). This transition between the two regimes is controled by the size of $\frac{\nu}{c} \frac{|\nabla u|}{u}$ indeed. Actually, there are many models that share roughly the same dynamics and only differ in the way they interpolate between (\ref{eq:he}) and (\ref{eq:trmedia}). A very general family of models of the form 
\begin{equation}
\label{eq:a_form}
\frac{\partial u}{\partial t} = \div ( \a(u,\nabla u))
\end{equation} 
 has been studied in \cite{Andreu2005elliptic,Andreu2005parabolic} for a flux function $\a(z,\xi)$ whose modulus saturates to a constant value when $|\xi|\to \infty$. These works show well-posedness for \eqref{eq:a_form} in the class of \emph{entropy solutions}. A convenient subfamily of \eqref{eq:a_form} is considered in \cite{Calvo2015}. 

To the best of our knowledge, most of the literature concerning tempered diffusion equations is devoted to models that are isotropic in the following sense: maximum pro\-pa\-gation speeds in the hyperbolic regime do not depend on the direction of $\nabla u$. However, there are natural situations in which one would like to consider spreading speeds that do depend on the specific direction the gradient of the solution points to. As an example, we recall here the one-dimensional models for morphogenesis that have been introduced in \cite{Verbeni2013} and analyzed in \cite{Andreu2012,Calvo2011}; in more realistic representations the actual extracellular medium should be described as a three-dimensional position space. In order to proceed with such generalizarions we have to consider the fact that the extracellular matrix is highly anisotropic and thus favors some privileged spreading directions locally. Some chemotaxis models based on flux-saturated spreading can acommodate this type of effects \cite{Chertock2012}. A similar situation takes place for brain tumors, where glioma cells tend to move preferentially along the directions given by white matter tracts -see e.g. \cite{Engwer}. Our aim in this document is to formulate an extension of the class in \cite{Calvo2015} to make sense of anisotropic spreading. Once this framework is well-grounded we will proceed to analyze the properties of these new families of equations.

Anisotropic diffusion equations constitute an active research topic, especially in connection with geometric flows and Finsler metrics \cite{Antontsev2009,Chambolle2017,Duzgun,Mazon2016,Moll2005,Wang2011}. Applications to image processing also abound, e.g. diffusion tensor imaging or Perona--Malik equations and denoising, see for example \cite{Aubert,Calderero,Chan,Esedoglu,Lasica,Soares} and references therein. Our work contributes to expand the scope of this area by considering anisotropic nonlinearities of tempered diffusion type. In passing, we contribute to simplify the original theory of \cite{Andreu2005elliptic,Andreu2005parabolic} (although here we consider a proper subclass of theirs) by deriving all relevant properties from a very synthetic and amenable hypothesis set on the model -compare also with the hypothesis list in \cite{Calvo2015}. 

Let us be more specific about the structure of this document. Next subsection is devoted to introduce specific notations and state our results. Section \ref{sec:frm} lays the functional framework to make sense of anisotropic tempered diffusion equations within the entropy solution framework set by \cite{Andreu2005elliptic,Andreu2005parabolic} and, as a consequence, the well-posedness of every model within the proposed framework will follow. A series of examples are presented at the end of the section. We state some connections with optimal transport theory in Section \ref{sec:optr}. In particular we derive the cost functions associated with anisotropic tempered spreading. Then Section \ref{sec:comp} provides a complementary view on the subject by characterizing spreading rates by means of comparison principles with suitable sub- and super-solutions. The goal of Section \ref{sec:RH} is to describe the temporal evolution of jump discontinuities, whose dynamics will strongly depend on their orientation with respect to the anisotropic spreading field. The document concludes with an appendix summarizing the functional framework needed to make sense of the class of {\em entropy solutions} to tempered diffusion equations.

\subsection{Summary of notational conventions and statement of the results}

We will work in the euclidean space $\R^d, \, d\ge 2$ endowed with the standard scalar product, that we denote as $u\cdot v$ for $u,v\in \R^d$. Unless otherwise stated, $u\in \R^d$ is regarded as a row vector and $u^T$ denotes its transpose vector. We use $|\cdot|$ to denote either the modulus of a vector or the absolute value of a number; this will be clear from the context. Given $r\in \R^d \backslash \{0\}$ we define the associated direction $\theta_r \in \mathbb{S}^{d-1}$ as $\theta_r:=r/|r|$. We write $u_i$ for the i-th component of $u\in \R^d$ and $a_{ij}$ for the $(i,j)$-element of a matrix $A$. The determinant of a square matrix $A$ will be denoted by $\det (A)$. The tensor product of $u,v \in \R^d$ is indicated by $u\otimes v$ and yields a $d\times d$ matrix. We recall that a symmetric $d\times d$ matrix $A$ is positive semidefinite (resp. positive definite), denoted as $A \ge 0$ (resp. $A > 0$), whenever $r A r^T \ge 0$ (resp. $r A r^T > 0$) for every $r\in \R^d\backslash \{0\}$. The Kronecker delta is $\delta_{ij}=1$ if $i=j$, zero otherwise. Einstein's summation convention {shall not} be used in this document.
 
We will denote the closure and interior of a given set $S$ by $\overline S$ and ${\rm int}\, S$ respectively. The complement of a set $S \subset \R^d$ is $S^c:=\{x\in \R^d/ x\notin S\}$. Let $B(p,r)$ stand for the open ball (and $\overline B(p,r)$ for the closed ball) centered at $p$ and having radius $r$. The Minkowski sum of two sets $A,B \subset \R^d$ is defined as 
$$
A \oplus B := \{x+y/ x\in A,\, y \in B\}\:.
$$
The dilation of a set $A$ by a positive number $\lambda$ is given by $\lambda A:= \{\lambda x/ x\in A\}$. Given a set $A \subset \R^d$ we let its indicator function be
$$
\chi_A(u):=\left\{
\begin{array}{c}
0 \quad \mbox{if}\, \, u\in A 
\\
+\infty \quad \mbox{otherwise} 
\end{array}
\right.
$$
and the characteristic function 
$$
\1_A(u):=\left\{
\begin{array}{l}
1 \quad \mbox{if}\, u\in A 
\\
0 \quad \mbox{otherwise}\:. 
\end{array}
\right.
$$
We shall also use the support function of a set $A \subset \R^d$,
$$
\sigma_A(u):= \sup_{u^* \in A}  u \cdot u^*\:, \quad u\in \R^d.
$$
Given $f$ a proper function defined in $\R^d$, we write its Legendre--Fenchel transform as
$$
f^*(y):= \sup_{x\in \R^d}\,  \{y \cdot x -f(x)\}\:, \quad y\in \R^d.
$$
We say that a function $f:\R^n\rightarrow \R^m$ is positively homogeneous/1-homogeneous whenever $f(\lambda x)= |\lambda| f(x)$ for every $\lambda \in \R$ and every $x\in \R^n$. We will use the notation $f_\infty$ to denote the associated recession function, defined as
$$
f_\infty : \S^{d-1} \rightarrow \R^m,\quad  f_\infty(\xi) = \lim_{t \to 0^+} t\, f(\xi/t).
$$
Note that $f_\infty$ is the support function of $f^*$ (see e.g. \cite{Rockafellar}, Th. 13.3) and hence
$$
{\rm dom}\, f^*=\{y\in \R^d/y\cdot x \le f_\infty(x)\, \forall x\in \R^d\}\:.
$$

Given an open set $\Omega \subset \R^d$ we denote by $\mathcal{D}(\Omega)$ the space of infinitely differentiable functions with compact support in $\Omega$. The space of continuous functions with compact support in $\Omega$ will be denoted by $C_c(\Omega)$. In a similar way, $L^p(\Omega)$ and $C^k(\Omega)$ denote Lebesgue spaces of p-integrable functions and spaces of functions of class k. Given a vector field $\psi \in C^1(\R^d;\R^d)$, we denote its Jacobian matrix as $D\psi$. Similarly, for $\Phi \in C^2(\R^d)$ we denote by $D^2\Phi$ its Hessian matrix. We use $\|\cdot\|_p$ to denote the norm in $L^p(\Omega)$, the base set will be clear from the context. Given $u : \Omega \rightarrow \R$, ${supp}\, u$ denotes the essential support. For any $T > 0$, we let $Q_T := (0, T ) \times \R^d$ and we write $u = u(t, x)$ for functions defined in $Q_T$. We recall that $O()$ and $o()$ are the standard Landau symbols, while $\sim$ indicates asymptotic equivalence.

Let $\LL^d$ and $\H^{d-1}$ stand for the $d$-dimensional Lebesgue measure and the $(d-1)$-dimensional Hausdorff measure respectively. We will indicate the supporting set for $\H^{d-1}$ as a subscript whenever necessary. We say that $u\in L^1(\Omega)$ is a function of bounded variation ($u\in BV(\Omega)$ for short) whenever its distributional derivative $Du$ is a vector-valued Radon measure with finite total variation in $\Omega$. We use $|Du|$ to denote the associated total variation measure. In a similar fashion, we say that $u\in BV_{loc}(\Omega)$ whenever $u\in BV(K)$ for every compact set $K\subset \Omega$. Recall that given $u\in BV(\Omega)$ we can decompose its derivative as $Du = D^{ac}u+D^su$ -absolutely continuous and singular parts. We have $D^{ac}u= \nabla u \LL^d$, being $\nabla u$ the Radon--Nikodym derivative of $Du$ with respect to $\LL^d$. We can split $D^su$ into the jump part $D^ju$ and the Cantor part $D^cu$. See e.g. \cite{Ambrosio2000} for more details.

\newpage
The main results of the document are summarized in the following statement.
\begin{theorem}
\label{th:1}
Consider an equation of the form \eqref{eq:anisotropic_template0}. Let $\psi = \nabla \Phi$ with $\Phi$ satisfying Assumptions \ref{ass:1} below. Then the following assertions hold true:
\begin{enumerate}

\item The Cauchy problem for \eqref{eq:anisotropic_template0} is well-posed in the sense of \cite{Andreu2005elliptic,Andreu2005parabolic}.

\item (Cost functions with bounded domain) Let $k:\R^d \rightarrow \R_0^+ \cup \{+\infty\}$ be a cost function defined via $k=\Phi^*$. Then $\overline{\mbox{dom}\, k}$ coincides with the dual unit ball (Wulff shape) of $\Phi_\infty$, the recession function of $\Phi$. Moreover, \eqref{eq:anisotropic_template0} can be recovered from the point of view of optimal mass transport problems as the (formal) limit of  the Jordan--Kinderlerher--Otto minimization scheme with the Boltzmann entropy and the Wasserstein distance associated with the cost function $k$.

\item (Evolution of the support) Let $u(t)$ be the  entropy solution of \eqref{eq:anisotropic_template0} with initial datum $u_0$. Assume that $\mbox{supp}\, u_0$ is either convex or has a boundary of class $C^2$. Then there holds that $\mbox{supp}\, u(t) \subset \mbox{supp}\, u_0 \oplus t\, \overline{\mathrm{dom}\, k}$ for any $t \ge 0$.

\item (Evolution of jump discontinuities) Let $u \in C([0,T];L^1(\R^d))$ be the entropy solution of \eqref{eq:anisotropic_template0} with $0\le u(0)=u_0 \in L^\infty(\R^d)\cap BV(\R^d)$. Assume that $u \in BV_{loc}(Q_T)$. Then the speed $\v(t,x)\in \R^d$ of any discontinuity front is given by
$$
\v=\Phi_\infty \left(\nu^{J_{u(t)}}\right)
$$
where $\nu^{J_{u(t)}}$ is the unit normal to the jump set of $u(t)$.

\end{enumerate}
\end{theorem}

We can give more specific information for models generated by classical anisotropic norms. 

\begin{proposition}
\label{pr:1}
Let $A$ be a symmetric matrix whose associated quadratic form is positive-definite. Define $\|x\|_A:=\sqrt{x A x^T}$. Let $ g:\R_0^+\rightarrow \R_0^+$ be a $C^1(\R_0^+)$-function such that the decay condition \eqref{eq:newdecay} below holds. Then Theorem \ref{th:1} holds for the equation
\begin{equation}
\label{eq:anisotropic_template3}
\frac{\partial u}{\partial t} = \div \left( g\left(\|\nabla u/u\|_A^2/2\right) A \nabla u\right).
\end{equation}
 Moreover, we have that:
 \begin{enumerate}
 \item $\overline{\mbox{dom}\, k} = {\{x\in \R^d / \|x\|_{ A^{-1}}\le 1\}}$.
 \item Assume that the enhanced decay condition \eqref{eq:extralimit} below holds. Then, for any compactly supported initial datum $u_0$ satisfying the nondegeneracy condition \eqref{ass:F} and such that $\mbox{supp}\, u_0$ is either convex or has a boundary of class $C^2$, there holds that 
 $$\mbox{supp}\, u(t) = \mbox{supp}\, u_0 \oplus t\bar B_{A^{-1}}.$$
 \end{enumerate}
\end{proposition}

\section{A framework for anisotropic tempered diffusion equations}
\label{sec:frm}
The goal of this section is to introduce a convenient framework that allows to make sense of anisotropic generalizations of equations like \eqref{eq:rhe0}, so that the resulting framework will allow us to prove qualitative solutions' properties for such generalizations. Here we understand the anisotropy in the sense that the maximum propagation speed may depend on the spatial direction. 

As far as we know, previous works impose a number of assumptions on the flux function defining the equation (be it $\psi$ in \eqref{eq:anisotropic_template0} or $\a$ in \eqref{eq:a_form}) and then study the properties of the resulting equation, e.g. \cite{Andreu2005parabolic,Calvo2015,Calvo2015survey}. Here we shall change the focus from the structure of the flux to the structure of an associated potential. Thus, we shall set a functional framework for suitable potential functions and study the equations that are generated by those potentials. Having these ideas in mind we display the following framework for the generating potentials:
\begin{assumptions}
\label{ass:1}
Let $\Phi:\R^d \mapsto \R_0^+$ a $C^2$-function such that:
	\begin{enumerate}
	\item $\Phi \ge 0$ is convex and even.
	\item 
	\label{limit1}
	The recession function of $\Phi$, $\Phi_\infty$, verifies that $\mathbb{S}^{d-1} \subset \mbox{dom}\, \Phi_\infty$.
\item $\min_{\mathbb{S}^{d-1}} \Phi_\infty> 0$.
	\end{enumerate}
\end{assumptions}
\begin{remark}
A couple of useful facts follow easily from the former hypotheses: (i) since $\Phi_\infty$ is 1-homogeneous, we have that actually $\mbox{dom}\, \Phi_\infty =\R^d$.
(ii) condition \emph{\ref{limit1}} above is equivalent to ask that $\Phi$ be globally Lipschitz (\cite{Rockafellar}, Corollary 13.3.3), which rules out superlinear growth. Thus, $\sup_{\mathbb{S}^{d-1}} \Phi_\infty<+\infty$, 
\end{remark}
The function $\Phi$ plays the role of a potential for the direction field $\psi$ in \eqref{eq:anisotropic_template0}. Our next result makes this connection clear. Just before that, let us introduce some auxiliary objects that are needed to use the entropy solution theory as formulated in \cite{Andreu2005elliptic,Andreu2005parabolic} -see also the appendix.
\begin{definition}
\label{df:list}
Let $\Phi$ a potential function satisfying assumptions \ref{ass:1} above and let $\psi :=\nabla \Phi$ the associated vector field. We introduce:
\begin{itemize}
\item The associated flux, ${\bf a}(z,\xi):=z \psi(\xi/z)$ -compare \eqref{eq:anisotropic_template0} with \eqref{eq:a_form}. 
\item The associated Lagrangian, $F(z,\xi):=z^2 \Phi(\xi/|z|)$. This is such that $F(z,0)=0$ and $\partial_\xi F={\bf a}$ -its is implied here that we choose $\Phi(0)=0$; by prescribing this we are not losing any generality.
\item The recession function of the Lagrangian with respect to the $\xi$-variable, computed as $ F_\infty(z, \xi) := \lim_{t \to 0^+} t\, F \left(z, \xi/t
 \right).$ In our framework we actually have $F_\infty(z,\xi):=z |\xi|\Phi_\infty(\theta_\xi)$.

 \item The auxiliary function $h(z,\xi):=\a(z,\xi) \cdot \xi$, which in our framework reads as follows: $h(z,\xi)=z \psi(\xi/z) \cdot \xi$. We also define $ h_\infty(z, \xi) := \lim_{t \to 0^+} t\, h \left(z, \xi/t \right).$
 
\end{itemize}
\end{definition}
\begin{lemma}
\label{lm:coherence}
Let $\Phi$ a potential function satisfying assumptions \ref{ass:1} above and consider $\psi :=\nabla \Phi$ the associated vector field. Let $L>0$ -characteristic lengthscale. Then the equation
\begin{equation}
\label{eq:anisotropic_template}
\frac{\partial u}{\partial t} = \div \left(u\,  \psi \left(L\frac{\nabla u}{u}\right)\right)
\end{equation}
falls under the framework of \cite{Andreu2005elliptic,Andreu2005parabolic} and hence it is well posed in the class of entropy solutions.
\end{lemma}
\begin{proof} 

{\em Step 1: preliminary observations.}
We are to show that (\ref{eq:anisotropic_template}) satisfies the hypotheses in \cite{Andreu2005elliptic,Andreu2005parabolic} and hence the well-posedness theory developed there will apply, thus granting the desired statement. Since our equation template resembles closely that in \cite{Calvo2015}, part of the arguments there can be reused for the present purpose. To that aim, we start by noting that $\psi \in C^1(\R^d;\R^d)$. This vector field is conservative and monotone ($D\psi \ge 0$) by construction. Moreover, being $\Phi$ even we can ensure that $\psi(0)=0$; we also have that $\psi(-r)=-\psi(r)$ since $\Phi(-r)=\Phi(r)$.

{\em Step 2: decay of the derivatives.} Let us write $\psi_i$, $i=1,\ldots,d$ for the components of $\psi$.  
Next we show that the map $r\mapsto \max_{i,j}|\partial_j \psi_i(r)|$ is an $O(1/|r|)$ for $|r|\gg 1$. Since $\psi$ is conservative we can represent $\Phi$ using Poincare's lemma, that is,
$$
\Phi(r)= \Phi(0)+ \int_0^1 \psi(tr) \cdot r \, dt,\quad \forall r \in \R^d.
$$
Thus
$$
\lim_{|r|\to \infty} \int_0^1 \psi(tr) \cdot \frac{r}{|r|} \, dt = \lim_{|r|\to \infty} \Phi(r)/|r| = \Phi_\infty(\theta_r)<\infty\:.
$$
Recall that we abridge $\theta_r := r/|r|$ for the direction of the ray determined by $r$. Now we note that, thanks to the monotonicity of $\psi$ (or equivalently $\frac{d}{dt}[r\cdot \psi(tr)]=r\, D\psi_{|(tr)} \, r^T\ge 0$) we have that $[0,\infty) \ni t\mapsto r\cdot \psi(tr)$ is a non-decreasing function that assumes the value zero at $t=0$ and attains strictly positive values for large $t$ -since $\Phi$ is not constant along any ray. Therefore the map $t\mapsto r\cdot \psi(tr)$ is always non-negative and the following limit exists: $lim_{t\to \infty}r\cdot \psi(tr) \in \R^+ \cup \{+\infty\}$. We now write
$$
\Phi_\infty(\theta_r) = \lim_{|r|\to \infty} \int_0^1 \theta_r \cdot \psi(tr) \, dt = \lim_{|r|\to \infty} \frac{1}{|r|} \int_0^{|r|} \theta_r \cdot \psi(z \theta_r)\, dz >0
$$
which implies that $\lim_{|r|\to \infty} |\psi(r)|$ is finite. {Even more, using the monotonicity of $z\mapsto \theta_r \cdot \psi(z \theta_r)$ we may sharpen the previous to yield}
\begin{equation}
\label{eq:psi_infty}
\lim_{r \to \infty} \theta_r \cdot \psi(r) = \Phi_\infty(\theta_r) \quad \mbox{and hence}\quad \lim_{r \to \infty} \frac{r \cdot \psi(r)}{|r| \Phi_\infty(\theta_r)} = 1. 
\end{equation}
We are now ready to analyze the derivatives of $\psi$: note that
$$
\psi(r) =\nabla \Phi(r)= \int_0^1 \psi(tr) + t\, r D\psi_{|(tr)}\, dt,
$$
which implies that
$$
\lim_{r\to \infty} \psi(r) \cdot \theta_r^T = \lim_{r\to \infty} \int_0^1 \psi(tr)\cdot \theta_r^T + t\, r D\psi_{|(tr)} \theta_r^T\, dt\:.
$$
We have just shown that the limit on the left hand side exists and is finite. We also know that both terms on the right hand side are non-negative; moreover, the integral of the first summand converges by the same token. Therefore,
$$
\lim_{r\to \infty} \int_0^1  t\, r D\psi_{|(tr)}\cdot \theta_r^T\, dt = \lim_{r\to \infty} \frac{1}{|r|} \int_0^{|r|} z \theta_r D\psi_{|(z \theta_r)} \theta_r^T\, dt<\infty,
$$
which finally implies that $r D\psi \theta_r^T=O(1)$ for large $|r|$ and thus the decay rate for $D\psi$ is obtained.

{\em Step 3: Fulfillment of the assumption list in  \cite{Andreu2005elliptic,Andreu2005parabolic}.}
All in all, we have shown that $\psi$ is a $C^1$ vector field which is conservative, monotone and odd. Moreover it vanishes at the origin and its derivatives decay at infinity. These properties enable us to argue like in Section 4 of \cite{Calvo2015} to justify that our equation satisfies the assumption set in \cite{Andreu2005elliptic,Andreu2005parabolic}. For that aim, we address the following points:
\begin{enumerate}

\item The continuity of $\frac{\partial \a}{\partial \xi_i}$ for each $i=1,\ldots, d$ is dealt with as in Lemma 4.1 of \cite{Calvo2015}. Here we use the decay of the derivatives that was proved in the previous step.

\item {\em (Continuity of the Lagrangian)} The developments in Step {\em 2} imply that
\begin{equation}
\label{eq:Phiexpand}
\Phi(r)= \Phi_\infty(\theta_r) |r| + o(|r|) \quad \mbox{for} \quad |r|\gg 1.
\end{equation}
 This entails that $F(z,\xi)$ can be made continuous at $z=0$ and thus $F\in C(\R \times \R^d)$.
 
 \item {\em (Bounds on the Lagrangian)} The upper bound $F(z,\xi) \le C(1 + |\xi|)$ follows easily as in \cite{Calvo2015}, Section 4. To obtain a lower bound we argue as follows. Thanks to \eqref{eq:Phiexpand} we can find some $\tilde r>0$ such that $\Phi(r) \ge |r| \min_{\mathbb{S}^{d-1}} \Phi_\infty/2$ for every $|r| >\tilde r$. Thus we have that 
$$
\Phi(r) \ge \frac{|r|}{2} \min_{\mathbb{S}^{d-1}} \Phi_\infty - \frac{|\tilde r|}{2} \min_{\mathbb{S}^{d-1}} \Phi_\infty \quad \forall r\in \R^d
$$
and finally
$$
F(z,\xi) \ge \frac{|z|}{2} |\xi| \min_{\mathbb{S}^{d-1}} \Phi_\infty - z^2 \frac{|\tilde r|}{2} \min_{\mathbb{S}^{d-1}} \Phi_\infty \quad \forall z\in \R,\,  \xi \in \R^d.
$$

\item {\em (Properties of $h$)} All the required properties follow easily but the lower bound on $h_\infty$. We need to ensure that 
$$
 \a(z,\xi) \cdot \eta \le h_\infty(z,\eta) \quad \mbox{for every}\quad \xi, \eta \in \R^d, \, \mbox{and}\, z \in \R.
$$
In our framework this is equivalent to
$$
z \psi(\xi/|z|)\cdot \eta \le z |\eta|\Phi_\infty(\theta_\eta)
$$
and it would be enough to show that, for every $\xi, \eta \in \R^d$ and $z \in \R$, we have
$$
 \psi(\xi/|z|)\cdot \theta_\eta \le \Phi_\infty(\theta_\eta)\:.
$$
Given that $\Phi_\infty = \sigma_{{\rm dom}\, \Phi^*}$, this would be implied by ${\rm Im}\, \psi ={\rm Im}\, \nabla \Phi \subset \mbox{dom}\, \Phi^*$. This last property readily follows from the convexity of $\Phi$, as Lemma \ref{lm:folklore} below shows.

\item {\em (Lipschitz bound for $|(\a(z,\xi)-\a(\hat z,\xi)) (\xi-\hat \xi)|$)} This is obtained as in Lemma 4.4 of \cite{Calvo2015}.

\end{enumerate} 
The above guidelines are enough to argue like in Section 4 of \cite{Calvo2015}, thus justifying that our equation satisfies the assumption set in \cite{Andreu2005elliptic,Andreu2005parabolic}, which finally grants us the well-posedness of the model in the class of {\em entropy solutions}.
\end{proof}
We have used the following particular version of a well-kown fact in convex analysis.
\begin{lemma}
\label{lm:folklore} Let $\Phi$ be as in Assumptions \ref{ass:1}. Then the subdifferential of $\Phi$ is contained in the domain of $\Phi^*$.
\end{lemma}
\begin{proof}
Since $\Phi \in C^2(\R^d)$ is convex we may fix $x_0 \in \R^d$ and write
$$
\Phi(x)-\Phi(x_0) \ge \nabla \Phi(x_0)\cdot (x-x_0)\quad \forall x \in \R^d,
$$
that is,  
$$
 \nabla \Phi(x_0)\cdot x-\Phi(x) \le x_0\cdot \nabla \Phi(x_0) - \Phi(x_0)<\infty \quad \forall x \in \R^d.
$$ 
Therefore $\nabla \Phi(x_0)$ belongs to $\mbox{dom}\, \Phi^*$. 
\end{proof}

\subsection{Examples}
Our next aim is to show some examples of genuine anisotropic tempered diffusion equations; those will be described thanks to the previous framework. For that we need some preliminary results.
\begin{definition}
Let $A$ 
be a $d\times d$ positive definite symmetric matrix.
For any $w,z\in\mathbb R^d$ we set $\langle w,z\rangle_A
=z A w^T$ and $\|z\|_{A}:=\sqrt{\langle z,z\rangle_A}$.
\end{definition}
\begin{lemma}
\label{lm:example}
Let $A$ be a symmetric matrix whose associated quadratic form is positive-definite. Let $ g:\R_0^+\rightarrow \R_0^+$ be a $C^1(\R_0^+)$-function such that
\begin{equation}
\label{eq:newdecay}
\lim_{z\to \infty}\sqrt{z}\, g(z)=1/\sqrt{2} \quad \mbox{and} \quad |z  g'(z)/ g(z)|\le 1/2, \quad z\in [0,\infty).
\end{equation}
Then the vector field
$\psi: \R^d \rightarrow \R^d$ given by
\begin{equation}
\label{eq:Apsi}
\psi(r)= g(\|r\|_A^2/2) A r^T
\end{equation}
is conservative and its associated potential $\Phi$ satisfies Assumptions \ref{ass:1}.
\end{lemma}
\begin{proof}
 Let $G:\R_0^+ \rightarrow \R_0^+$ be the primitive of $g$ that vanishes at zero. We let 
$$
\Phi:\R^d \rightarrow \R_0^+,\quad \Phi(r):=  G (\|r\|_A^2/2).
$$
We readily check that $\nabla \Phi = \psi$, where we use the fact that $A$ is symmetric. Let us see that the potential $\Phi$ constructed in this way satisfies Assumptions \ref{ass:1}. It is an even function by construction. We can compute the recession function as per
$$
\frac{\Phi(r)}{|r|} =\frac{1}{|r|}\int_0^{\|r\|_A^2/2}  g(\xi)\, d\xi \sim \frac{1}{|r|}\int_1^{\|r\|_A^2/2}  \frac{d\xi}{\sqrt{2\xi}} \sim \frac{1}{|r|} \sqrt{\|r\|_A^2}\quad \mbox{for}\, |r|\gg 1.
$$
Thus, letting $\theta_r:=r/|r|$, we have  
\begin{equation}
\label{eq:reccA}
\Phi_\infty(\theta_r)= \|\theta_r\|_A.
\end{equation}

Next we compute the Jacobian matrix
$$
D\psi = g'(\|r\|_A^2/2) (Ar^T) \otimes (Ar^T) + g(\|r\|_A^2/2)A
$$
and hence, given that $A$ is symmetric we deduce that $D\psi$ also is. Since $g\in C^1(\R_0^+)$ we have $\Phi \in C^2(\R_0^+)$.
 It remains to prove the convexity of $\Phi$, which is equivalent to the monotonicity pro\-per\-ty $D\psi \ge 0$. As $A$ defines a quadratic form which is positive definite, there is no loss of generality in assuming that $A$ is a diagonal matrix with strictly positive diagonal elements. We may use Sylvester's criterion as in \cite{Calvo2015}. For $i,j \in \{1,\ldots,d\}$ we have  
 $$
\partial_j \psi^{(i)} =  g'(\|r\|_A^2/2) a_{ii}r_i a_{jj}r_j + g(\|r\|_A^2/2) \delta_{ij} a_{ij}.
$$
To computer the principal k-minors we use the following matrix identity (known as ``Sylvester's trick'') 
$$
\det (B+u\otimes v)= \det B \, (1+v B^{-1} u^T)\:;
$$
 e.g., to compute $\det(D\psi)$ we can take $B=  g(\|r\|_A^2/2) A,\, u_i=a_{ii}r_i,\, v_i=  g'(\|r\|_A^2/2) u_i$. For a $d\times d$ matrix $A$ we denote by $A_{k\times k}$ the $k\times k$ matrix $\tilde A$ given by $\tilde a_{ij}=a_{ij}$, $i,j\in \{1,\ldots,k\}$. Then the principal k-th minor reads
\begin{equation}
\label{eq:kminor}
\det[(D\psi)_{k\times k}] =  g(\|r\|_A^2/2)^k \det[(A)_{k\times k}] \left(1 +\frac{ g'(\|r\|_A^2/2)}{ g(\|r\|_A^2/2)} \sum_{i=1}^k a_{ii} (r_i)^2\right).
\end{equation}
We produce a set of $d$ conditions by imposing \eqref{eq:kminor} to be non-negative for each $k=1,\ldots,d$. Clearly the most demanding condition is that for $k=d$, which reads
$$
\frac{| g'(\|r\|_A^2/2)|}{ g(\|r\|_A^2/2)}\frac{\|r\|_A^2}{2} \le \frac{1}{2}.
$$
This is satisfied thanks to \eqref{eq:newdecay}. In this way we ensure that $D\psi \ge 0$ and we conclude the proof.
\end{proof}
Now we can present specific examples. Following \cite{Olson2000} we may consider
\begin{equation}
\label{eq:pmodel}
g(z)= \frac{1}{(1+(2z)^{p/2})^{1/p}},\quad p\ge 2.
\end{equation}
This includes \eqref{eq:rhe0} for $p=2$ and Wilson's model for $p=1$ (e.g. \cite{Mihalas1984}, see also Example \ref{ex:Wilson} below). Note that $g$ is not a $C^1$ function at the origin when $1\le p<2$. Nevertheles, it can be shown that $\Phi \in C^2$ for any $1\le p<2$ and hence the thesis of Lemma \ref{lm:example} is also true for this range of $p$. In the isotropic case we would construct a flux as $\psi(r)=r g(|r|^2/2)$ -see Remark \ref{rm:isog} below. Now we obtain anisotropic variants of those models by replacing that flux by the one given in \eqref{eq:Apsi} for a suitable matrix $A$. In this way it makes perfect sense to consider equations as
$$
\frac{\partial u}{\partial t} = \div \left[ \frac{u\, A \nabla u}{\sqrt{u^2 + (\nabla u)^TA\nabla u}} \right] \quad \mbox{in}\, Q_T
$$
or more generally
$$
\frac{\partial u}{\partial t} = \div \left[ \frac{u\, A \nabla u}{\left(u^p + [(\nabla u)^TA\nabla u]^{p/2}\right)^{1/p}} \right] \quad \mbox{in}\, Q_T.
$$

\begin{remark}
\label{rm:isog}
The isotropic framework described in \cite{Calvo2015} is embodied here -although notations differ slightly. Starting from the class of vector fields $\psi$ described in \cite{Calvo2015} we can construct the associated potential $\Phi$ using Poincare's lemma as in \cite{Calvo2015}; this potential is easily seen to satisfy Assumptions \ref{ass:1}. Actually we have $\Phi_\infty(r)=s|r|$ -where $s$ is the sound speed according to the notations in \cite{Calvo2015}- regardless of the direction we approach infinity.
\end{remark}

\section{Connections with optimal transport theory}
\label{sec:optr}
In order to consolidate our framework we state and prove some results in the ambient of optimal transport theory. We relate the shape of the level sets of the potential to the maximum propagation speed allowed along each direction, thus making sense of anisotropic spreading at finite rates.

Recall that the Jordan--Kinderlerher-Otto (JKO for short) minimization scheme \cite{Jordan1998} with cost $k$ and entropy $F$ leads to evolution equations of the form
$$
u_t = \div\, (u \nabla k^* (\nabla F'(u))).
$$
Choosing the entropy $F(r)=L (r\log r - r)$ we arrive to the following structure:
\begin{equation}
\label{templatetr2}
u_t = \div\, (u \nabla k^* (L \nabla u/u)). 
\end{equation}
We wonder whether it is feasible to formulate equations of the form \eqref{eq:anisotropic_template} in this fashion and what would the properties of the cost function $k$ be. Comparing \eqref{eq:anisotropic_template} with \eqref{templatetr2} we readily see that $k^*(r) =  \Phi (r)$ for every $r \in \R^d$. Actually we can deepen into this observation as the following result shows.
\begin{proposition}
\label{prop:Phinorm}
Let $\Phi$ satisfy Assumptions \ref{ass:1} and define an associated cost function via $k=\Phi^*$. Then the following assertions hold true:
\begin{enumerate}
\item $\Phi_\infty$ is the support function of $\mathrm{dom}\, k$ (i.e. $\Phi_\infty=\sigma_{\mathrm{dom}\, k}$).
\item $\mathrm{dom}\, k$ is a convex and bounded set that is symmetric with respect to the origin. 
\item $\mathrm{dom}\, k$ has nonempty interior and $\Phi_\infty$ defines a norm on $\mathbb R^d$.
\end{enumerate}
\end{proposition}
\begin{proof}
This follows from well-know results in convex analysis, see e.g. \cite{Rockafellar}. Since $\Phi_\infty$ is the support function of $\mathrm{dom}\, \Phi^*$, we have the representation formula 
$\Phi_\infty=\sigma_{\mathrm{dom}\, k}=\sigma_{\overline{\mathrm{dom}\, k}}$; hence $\overline{\mathrm{dom}\, k}=
\{y\in \mathbb R^d / y \cdot x \le \Phi_\infty(x)\, \forall x \in \mathbb R^d\}$. We also have that $\Phi_\infty^*=\chi_{\overline{\mathrm{dom}\, k}}$ 
(then $\overline{\mathrm{dom}\, k}=\{y\in \mathbb R^d / \Phi_\infty^*(y)=0\}$).

Given that $\Phi_\infty$ is finite, we deduce that $\mathrm{dom}\, k$ is bounded. Recall next that $\Phi$ is even, then we have that $\Phi^*$ 
is also even and thus $\mathrm{dom}\, k$ is symmetric with respect to the origin. 
Since $x=0$ is the global minimum of $\Phi$, we have that $0\in \mathrm{dom}\, \Phi^* = \mathrm{dom}\, k$. 
This also implies that $\mathrm{dom}\, \Phi^*$ is balanced (meaning that $\alpha \mathrm{dom}\, \Phi^* \subset \mathrm{dom}\, \Phi^*$ for every $\alpha \in [-1,1]$). 
Besides, note that the level sets of $\Phi$ are bounded, which implies that $\mathrm{dom}\, k$ has nonempty interior and moreover $0 \in \mbox{int dom}\, k $.
\end{proof}
We can characterize the domain of the cost function in the following way:
\begin{proposition}
\label{pr:costasnorm}
Let $\Phi$ satisfy Assumptions \ref{ass:1}. 
Define an associated cost function via $k=\Phi^*$. Then $\overline{\mbox{dom}\, k}$ coincides with the Wulff shape of $\Phi_\infty$, i.e. the dual unit ball of $\Phi_\infty$.
\end{proposition}
\begin{proof}
During this proof we set $\|x\|:=\Phi_\infty(x)$, which defines a norm thanks to Proposition \ref{prop:Phinorm}. Its unit ball is given by the set $\{y \in \R^d /\Phi_\infty(y) \le 1\}$. Now we define the dual norm
$$
\|y\|_*:= \sup_{x\in \R^d\backslash \{0\}} \frac{y\cdot x}{\|x\|} = \sup_{x\in \R^d\backslash \{0\}} \frac{y\cdot x}{\Phi_\infty(x)}\:.
$$
Taking into account the representation $\overline{\mbox{dom}\, k}=\{y\in \R^d / y \cdot x \le \Phi_\infty(x)\, \forall x \in \R^d\}$ we deduce that $\overline{\mbox{dom}\, k}$ agrees with the unit ball for the norm $\|\cdot \|_*$. 
\end{proof}
We show by means of an example that taking the closure in the former result cannot be dispensed with, i.e. there are relevant cases where the domain of the cost function is not closed. Note in passing that our framework includes the isotropic case, as already hinted in Remark \ref{rm:isog}.
 \begin{example}
 \label{ex:Wilson}
 Wilson's model is already mentioned in \cite{Mihalas1984}. Normalizing all physical constants, it reads
 $$
\frac{\partial u}{\partial t} = \div \left( \frac{u\, \nabla u}{u + |\nabla u|} \right) \quad \mbox{in}\, \, \, Q_T
$$
We readily identify
\[
\nabla k^*(x)=\frac{x}{1+|x|}
\]
and hence
\[
\Phi(x)=k^*(x)=|x|-\log(1+|x|)\]
Moreover, the associated recession function can be computed as follows
\begin{align*}
\Phi_{\infty}(0)=&\lim_{t\to 0^+} t\Phi\left(0\right)=0,\\
\Phi_{\infty}(x)=&\lim_{t\to 0^+} t\Phi\left(\dfrac{x}{t}\right)=\lim_{t\to 0^+} t\left(\left|\dfrac{x}{t}\right|-\log\left(1+\left|\dfrac{x}{t}\right|\right)\right)\\
=&|x|-\lim_{t\to 0^+}t\log\left(1+\left|\dfrac{x}{t}\right|\right)\underset{\substack{x\ne 0\\s=\frac1t}}{=}|x|-\lim_{s\to +\infty}|x|\dfrac{\log\left(1+\left|xs\right|\right)}{|x|s}=|x|.
\end{align*}
Note also that
\[
\Phi^*_{\infty}(x)=\sup_{v\in\mathbb R^d}\langle v,x\rangle-|v|=\sup_{\lambda \ge 0}\lambda|x|-\lambda=\chi_{\overline{B(0,1)}}(x).
\]
In particular, we have that $\Phi_{\infty}$ is a norm.
We now compute the Legendre transform
\begin{align*}
k(x)=\Phi^*(x)=&\sup_{\xi\in\mathbb R^d}\left\{\langle \xi,x\rangle -|\xi|+\log(1+|\xi|)\right\}\\
=&\sup_{\substack{\xi=\lambda v\\ v\in\mathbb{S}^{d-1}\\ \lambda \ge 0}}\left\{\lambda\langle v,x\rangle -\lambda+\log(1+\lambda)\right\}\\
=&\sup_{\lambda\ge 0}\left\{\lambda(|x|-1)+\log(1+\lambda)\right\}
\end{align*}
We notice that if $|x|\ge 1$ immediately $k(x)=+\infty$, otherwise the sup is actually a maximum
achieved at $\lambda=\dfrac{|x|}{1-|x|}$ leading to
\begin{align*}
k(x)=\begin{cases}+\infty,&\textrm{ if }|x|\ge 1,\\ \\
-|x|+\log\left(\dfrac{1}{1-|x|}\right),&\textrm{ if }|x|< 1.
\end{cases}
\end{align*}
We check that $\mathrm{dom}\,k=B(0,1)$ and $\overline{\mathrm{dom}\,k}=\overline{B(0,1)}$, thus 
\[\sigma_{\mathrm{dom}\,k}(x)=\sigma_{\overline{\mathrm{dom}\,k}}(x)=\sup_{|v|< 1}\langle v,x\rangle=\sup_{|v|\le 1}\langle v,x\rangle=|x|.\]
Moreover $\Phi^*_{\infty}=\chi_{\overline{B(0,1)}}=\chi_{\overline{\mathrm{dom}\,k}}$.
 \end{example}
 Next we present an example of a truly anisotropic situation (models induced by anisotropic norms) where we can compute the closure of the domain of the associated cost functions; an anisotropic extension of the former Example \ref{ex:Wilson} fits here in a natural way.
 \begin{example}
 \label{ex:matrix_cost}
 {\rm 
Let $\psi$ be defined in terms of a positive definite symmetric matrix $A$ as in Lemma \ref{lm:example}. We can  compute $\overline{\mbox{dom}\, k}$ in terms of the properties of $A$. For that aim, let $A$ 
be a $d\times d$ positive definite symmetric matrix.
Recall that for any $w,z\in\mathbb R^d$ 
we set $\langle w,z\rangle_A
=z A w^T$ and $\|z\|_{A}:=\sqrt{\langle z,z\rangle_A}$.
Notice that $\langle w,z\rangle_A$ is a scalar product on $\mathbb R^d$, and its associated norm is equivalent to the Euclidean norm on $\mathbb R^d$.
We denote by $\overline{B}_A:=\{x\in\mathbb R^d:\,\|x\|_A\le 1\}$ the closed unit ball centered at the origin for the norm $\|\cdot\|_A$.
\par\medskip\par
Set $\Phi_A(z)=G\left(\|z\|^2_A/2\right)$ where $G:[0,+\infty)\to [0,+\infty)$ is a $C^2$ function satisfying $G(0)=0$ and assume that $\Phi_A(\cdot)$ is convex. We compute the associated recession function (denoted here by $\Phi_A^{\infty}$) as follows. For all $z\in\mathbb R^d$, $z\ne 0$ we have
\begin{align*}
\Phi_A^{\infty}(z)=&
\lim_{\lambda\to +\infty}\dfrac{\Phi_A(\lambda z)}{\lambda}=\lim_{\lambda\to +\infty}\dfrac{G\left(
\lambda^2\|z\|^2_A/2\right)}{\lambda}\\
=&\|z\|_A\lim_{\lambda\to +\infty} \dfrac{G\left(
\lambda^2\|z\|^2_A/2\right)}{\lambda\|z\|_A}=\|z\|_A\lim_{R \to +\infty}\dfrac{G\left(
R^2/2\right)}{R},
\end{align*}
while $\Phi_A^\infty(0)=0$. Hence the sublevel sets of $\Phi_A^{\infty}(\cdot)$ are the balls $\rho \overline{B}_A$, $\rho\ge 0$, of the norm $\|\cdot\|_A$.
By applying the l'H\^opital rule, we have 
\begin{align*}
\ell:=
\lim_{R\to +\infty}\dfrac{G\left(
R^2/2\right)}{R}=\lim_{R\to +\infty}R G'\left(
R^2/2\right)=\sqrt{2}\cdot\lim_{r\to +\infty} \sqrt r G'\left(r\right).
\end{align*}
If we assume, as in the first of formulas in (\ref{eq:newdecay}),  
\[\lim_{r\to +\infty} \sqrt r G'\left(r\right)=1/\sqrt{2},\]
we obtain $\Phi_A^{\infty}(z)= \|z\|_A$, same as in \eqref{eq:reccA}.
\par\medskip\par
We have 
\begin{align*}
\Phi_A^*(Aw^T)=&\sup_{z\in\mathbb R^d}\left\{ zAw^T-G\left(
\|z\|^2_A/2\right)\right\}=\sup_{z\in\mathbb R^d}\left\{\langle w,z\rangle_A-G\left(
\|z\|^2_A/2\right)\right\}\\
=&\sup_{R\ge 0}\sup_{\|z\|_A=R}\left\{\langle w,z\rangle_A-G\left(
 R^2/2\right)\right\}=\sup_{R\ge 0}\left\{R\cdot\|w\|_A-G\left(
  R^2/2\right)\right\}\\
=&\max\left\{0,\sup_{R>0}R\cdot\left\{\|w\|_A-\dfrac{G\left(
 R^2/2\right)}{R}\right\}\right\}\:.
\end{align*}
Taking into account the computation of $\ell$ we readily see that $\Phi_A^*(Aw^T)$ diverges for $\|w\|_A>\ell$ and remains finite for $\|w\|_A<\ell$. Therefore, 
\begin{align*}
\overline{\mathrm{dom}\,k}=\mathrm{dom}\,\Phi_A^*\subset&\left\{y\in\mathbb R^d:\,\|A^{-1}y\|_A\le\ell\right\}\\=&\left\{y\in\mathbb R^d:\,\|y\|_{A^{-1}}\le\ell\right\}=\ell\overline{B}_{A^{-1}}.
\end{align*}
By the same token,  
$$
\ell {B}_{A^{-1}} \subset \overline{\mathrm{dom}\,k} \subset \ell\overline{B}_{A^{-1}}\:.
$$
 Therefore, $\overline{\mathrm{dom}\,k}$ coincides with $\ell\overline{B}_{A^{-1}}$. Working as in \cite{Calvo2015} we may show that the cost function will be finite at $\partial (\ell B_{A^{-1}})$ provided that the map $z \mapsto \ell- z G'(z^2/2)$ is integrable at infinity.
}
 \end{example}

We can also approach the former developments the other way around: we may define a cost function supported on a convex set and write down the associated evolution model. 
\begin{proposition}
Fix a norm $\|\cdot\|_{k}$ on $\mathbb R^d$ and set ${B}_k:=\{y\in\mathbb R^d:\, \|y\|_{k}< 1\}$. Let $h:\mathbb R^d\to\mathbb R\cup\{+\infty\}$ be proper, strictly convex, lower semicontinuous and even with $\overline{\mathrm{dom}\,h}= \overline{B}_k$ and $h(0)=0$. Assume further that $h\in C^2(B_k)$ with $D^2\,h$ positive definite
and 
\[
\lim_{\substack{\|x\|_k\to 1^-\\ x\in B_k}}|\nabla h(x)|=+\infty.
\]
Define the cost
\[k(x)=\begin{cases}h(x),&\textrm{ if }\|x\|_{k}\le 1,\\ +\infty,&\textrm{ otherwise}.\end{cases}\]
Then, the anisotropic tempered diffusion model \eqref{templatetr2} is well-posed.
\end{proposition}
\begin{proof}
Under the present assumptions we have that $(B_k,h)$ is a Legendre pair in the sense of Theorem 26.5 of \cite{Rockafellar}. Moreover, by the smoothness of $h$ and the nondegeneracy
of its Hessian, we obtain that for all $p\in\mathrm{int}\,\mathrm{dom}\,h^*=\mathbb R^d$
\[D^2h^*(p)=\left[D^2h\left((\nabla h)^{-1}(p)\right)\right]^{-1},\]
and so $h^*\in C^2(\mathbb R^d)$, since $\nabla h$ and $D^2h$ are diffeomeorphisms. In this way $\Phi \in C^2(\R^d)$ and the rest of the proof follows easily from here.
\end{proof}
The existence and uniqueness of optimal transport maps for relativistic costs as those in the former statement has been analyzed in \cite{Bertrand2013,Chen}, whereas the existence of Kantorovich potentials was the subject of \cite{BertrandOtro}. All these contributions constitute a promising first step towards an anisotropic version of the results by \cite{McCann2009}, which stated the rigorous convergence of the JKO scheme for cost functions supported on a ball towards a tempered diffusion evolution equation.

\section{Comparison principles}
\label{sec:comp}
The goal of this section is to translate the information about finite propagation speed for (\ref{eq:anisotropic_template}) coming from the analysis of the cost function into specific results describing the evolution of the support of  actual solutions of (\ref{eq:anisotropic_template}). We do this by means of comparison principles with suitable sub- and super-solutions. The funtional framework to be used in this section is given at the appendix in Section \ref{sec:appendix}.

\subsection{Supersolutions, upper bounds on the spreading rate}
We shall show that compactly supported initial data launch solutions whose support is contained in the dilation of the original support by the Wulff shape of $\Phi_\infty$. For that aim we construct specific supersolutions. 
\begin{proposition}
\label{prop:super}
Let $\beta>0$ and let $K\subset \R^d$ be a compact set, which is either convex or has a boundary of class $C^2$. Let Assumptions \ref{ass:1} be satisfied. Let $K(t):=K \oplus t\, \overline{\mathrm{dom}\, k}$. Then $W(t,x)=\beta \1_{K(t)}$ is a super-solution of (\ref{eq:anisotropic_template}) in $Q_T$ for every $T>0$.
\end{proposition}
\begin{remark}
The regularity assumption on $K$ is of a technical nature (it enables us to compute the normal flow induced by $K(t)$ without having to deal with fine geometrical details). We expect the result to be true under more general hypotheses (probably Lipschitz regularity for the boundary would just do), although this would require technical developments which are out of the scope of the paper. The same comment is true for the statements in Theorem \ref{th:1}, Proposition \ref{pr:1} and Corollary \ref{cor:47}. At any rate, note that the set $K$ can always be included in a slightly larger set that complies with our current regularity assumption -see Corollary \ref{cr:sc}; thus, we conjecture that our results about support spreading are close to optimal.
\end{remark}
\begin{proof}
We have to check the fulfillment of (\ref{seineqbis}), see Section \ref{sec:appendix}. Actually we will show that the inequality will hold for the absolutely continuous and singular parts of the spatial derivative separately. To deal with the absolutely continuous part, first we notice that $\a(W,\nabla W)=0$ and hence $\div \a(W,\nabla W)=0$. On the other hand, we can study the time derivative using Reynolds' transport theorem, which states that
\begin{equation}
\label{eq:Reynolds}
\frac{d}{dt} \int_{\Omega(t)} f \, dx = \int_{\Omega(t)} \frac{\partial f}{\partial t}\, dx + \int_{\partial \Omega(t)} v^b \cdot \nu f \, d\sigma\:.
\end{equation}
Here $\nu$ is the exterior normal vector to $\partial \Omega(t)$ and $v^b$ is the velocity of the transport field $\Omega(t)$ distorting the shape of $\Omega(t)$ as time advances. As a consequence, we can show that $\partial_t J_{TS}(W)$ is a positive measure and thus, the absolutely continuous terms in (\ref{seineqbis}) are fine. To prove the former claim on $\partial_t J_{TS}(W)$, we compute 
\begin{equation}
\label{eq:Wt}
W_t = \beta |\kappa(t,x)| \H_{\partial K(t)}^{d-1}\quad \mbox{and likewise} \quad \partial_t J_{TS}(W)= J_{TS}(\beta) |\kappa(t,x)| \H_{\partial K(t)}^{d-1}
\end{equation}
where $\kappa(t,x)=v^b \cdot \nu$. Here we used \eqref{eq:Reynolds} and also the fact that $J_{TS}(W)= J_{TS}(\beta) \1_{K(t)}$.

Let us compute now the singular terms. Recall the notations for truncation functions in Section \ref{sec:appendix}. Set
$$
T:= T_{a,b}^{l_1}\ge 0,\quad S:=T_{c,d}^{l_2}\le 0 \quad \mbox{with}\, \, l_1,l_2 \in \R, \, 0<a<b\, \,\mbox{and}\, \, 0<c<d\, .
$$
Following \cite{Andreu2006} we get
\begin{equation}
\label{eq:follow5}
[h_S(W,DT(W))]^s=-\Phi_\infty\left( \frac{DT_{a,b}(W)}{|DT_{a,b}(W)|}\right) \left|D^s J_{(-S)I}(T_{a,b}(W)) \right|
\end{equation}
where $I$ denotes the identity function $I(z)=z$. Using again the results in \cite{Andreu2006} we compute
$$
\left|D^s J_{(-S)I}(T_{a,b}(W)) \right| =  J_{(-S)IT'}(\beta) \H_{\partial K(t)}^{d-1}= - J_{SIT'}(\beta) \H_{\partial K(t)}^{d-1}\:.
$$
Thus
\begin{equation}
\label{eq:S_T}
[h_S(W,DT(W))]^s= \Phi_\infty\left( \frac{DT_{a,b}(W)}{|DT_{a,b}(W)|}\right) J_{SIT'}(\beta) 
\end{equation}
and likewise
$$
[h_T(W,DS(W))]^s= \Phi_\infty\left( \frac{DS_{c,d}(W)}{|DS_{s,d}(W)|}\right) J_{TIS'}(\beta)\:.
$$
To ellaborate on \eqref{eq:S_T} we note that $T_{a,b}(W)=a \1_{K(t)^c} + T_{a,b}(\beta) \1_{K(t)}$. Let us denote the exterior normal vector to $K(t)$ at $x\in \partial K(t)$ by $\nu^{\partial K(t)}(x)$. Then we readily check that
$$
D T_{a,b}(W) = (T_{a,b}(\beta)-a) \nu^{\partial K(t)}(x) \H_{\partial K(t)}^{d-1}\:.
$$
Thus, provided that $\beta >a$ we get
$$
\frac{D T_{a,b}(W)}{|D T_{a,b}(W)|}= \nu^{\partial K(t)}(x) \, \, \mbox{and hence}\, \, [h_S(W,DT(W))]^s=\Phi_\infty\left(\nu^{\partial K(t)}(x) \right) J_{SIT'}(\beta) \H_{\partial K(t)}^{d-1}\:.
$$
This goes in the same way for $[h_T(W,DS(W))]^s$. Now we discuss the different cases that arise depending on the values of $\beta,a$ and $c$.

{\em Case $\beta>a$ and $\beta>c$:} Putting all computations so far together, we deduce that the singular part of (\ref{seineq}) can be recast as
\begin{eqnarray}
\label{eq:estrella}
\int_0^T \Phi_\infty\left(\nu^{\partial K(t)}(x) \right) \left\{J_{SIT'}(\beta)+ J_{TIS'}(\beta)\right\} \int_{\partial K(t)} \phi(t) \, d\H^{d-1} \, dt
\\
\nonumber
 \le -\int_{Q_T} \phi(t) \partial_t J_{TS}(W)\, dx dt 
\end{eqnarray}
and this must hold for every $\phi \in \mathcal{D}(Q_T)$. Taking into account that $TIS'+SIT'=I(TS)'$ and integrating by parts, we further simplify the left hand side of \eqref{eq:estrella} to
$$
\int_0^T \Phi_\infty\left(\nu^{\partial K(t)}(x) \right) \left\{\beta (TS)(\beta) -J_{TS}(\beta)\right\} \int_{\partial K(t)} \phi(t) \, d\H^{d-1} \, dt\:.
$$
 Using \eqref{eq:Wt} we reduce ourselves to check whether the following inequality holds true for every $\phi \in \mathcal{D}(Q_T)$:
\begin{eqnarray}
\label{eq:46}
\int_0^T \Phi_\infty\left(\nu^{\partial K(t)}(x) \right) \left\{\beta (TS)(\beta) -J_{TS}(\beta)\right\} \int_{\partial K(t)} \phi(t) \, d\H^{d-1} \, dt
\\ 
\nonumber
\le -\int_0^T J_{TS}(\beta) \int_{\partial K(t)} \phi(t) |\kappa(t,x)| \, d\H^{d-1} \, dt\:.
\end{eqnarray}
 We denote $E=\overline{\mathrm{dom}\, k}$ for short in what follows. Given $x\in \partial K(t)$ we can write $x=x_1+x_2$ with $x_1 \in \partial K$ and $x_2\in t \partial E$.
Here we use that $E$ is symmetric with respect to the origin thanks to Proposition \ref{prop:Phinorm}. Due to the regularity assumption on $K$, this decomposition is unique for a.e. $(t,x) \in Q_T$. Then, to first order in $t$, the trajectory traced by $x$ as we dilate $K$ by $tE$ is given by $x_1+t\,r^E(x_2)$, being $r^E(x_2)$ the radius vector of the point lying at the intersection of $\partial E$ with the ray determined by $x_2$. In that fashion, 
  $v^b(x)=r^{E}(x_2)$. Given that $(TS)(\beta)\le 0$, a sufficient condition for \eqref{eq:46} to hold is:
$$
\nu^{\partial K(t)}(x)\cdot r^{E}(x_2) \le \Phi_\infty(\nu^{\partial K(t)}(x))\quad a.e.\, \, (t,x) \in Q_T\:.
$$
Since $E=\overline{\mathrm{dom}\, k}$ is a closed set we have that $\partial E$ is contained within $E$ and more specifically we have $r^{E}(x_2)\in \overline{\mathrm{dom}\, k}$. Hence the above inequality relating $\nu^{\partial K(t)}(x)$ and $r^{E}(x_2)$ is authomatically satisfied as $\Phi_\infty=\sigma_{\overline{\mathrm{dom}\, k}}$. This implies that \eqref{eq:estrella} is finally satisfied.

{\em Case $\beta \le a$ and $\beta \le c$:} We note that when $\beta \le a$ the multiplicative factor $\left|D^s J_{(-S)I}(T_{a,b}(W)) \right|$ in \eqref{eq:follow5} vanishes. Hence $[h_S(W,DT(W))]^s$ does not contribute to \eqref{seineq}; the same happens with $[h_T(W,DS(W))]^s$. Therefore the left hand side of \eqref{seineqbis} vanishes identically, whereas the right hand side is non-negative thanks to \eqref{eq:Wt} and the fact that $J_{TS}(\beta)\le 0$, which concludes the proof in this case. 

{\em Case $\beta > a$ and $\beta \le c$:} We have to check the validity of the following inequality:
\begin{eqnarray}
\nonumber
\int_0^T \Phi_\infty\left(\nu^{\partial K(t)}(x) \right) J_{SIT'}(\beta) \int_{\partial K(t)} \phi(t) \, d\H^{d-1} \, dt
\\
\nonumber
\le -\int_0^T J_{TS}(\beta) \int_{\partial K(t)} \phi(t) |\kappa(t,x)| \, d\H^{d-1} \, dt\:.
\end{eqnarray}
Given that both $J_{SIT'}(\beta)$ and $J_{TS}(\beta)$ are non-positive this follows inmediately.

{\em Case $\beta \le a$ and $\beta > c$:} Now we would have to check the validity of 
\begin{eqnarray}
\nonumber
\int_0^T \Phi_\infty\left(\nu^{\partial K(t)}(x) \right) J_{TIS'}(\beta) \int_{\partial K(t)} \phi(t) \, d\H^{d-1} \, dt
\\
\nonumber
\le -\int_0^T J_{TS}(\beta) \int_{\partial K(t)} \phi(t) |\kappa(t,x)| \, d\H^{d-1} \, dt\:.
\end{eqnarray}
We write 
$$
J_{TIS'}(\beta)= \beta (TS)(\beta) - J_{TS}(\beta) -J_{SIT'}(\beta).
$$
The term in $-J_{TS}(\beta)$ is compensated with the right hand side as explained in the case $\beta > a$ and $\beta > c$;
the other two terms are non-positive. Thus, the inequality is fulfilled.
\end{proof}
\begin{corollary}
\label{cr:sc}
(support control). Let $0\le u_0\in L^\infty (\R^d)$ be compactly supported. Let $u(t)$ be the  entropy solution of \eqref{eq:anisotropic_template} with initial datum $u_0$. Consider a compactly supported set $K$ which is either convex or with a boundary of class $C^2$, such that $\mbox{supp}\, u_0 \subset K$. Then there holds that
\begin{equation}
\label{eq:supsubset}
\mbox{supp}\, u(t) \subset K \oplus t\, \overline{\mathrm{dom}\, k}\quad \forall t \ge 0.
\end{equation}
\end{corollary}

\subsection{Subsolutions}
We aim to prove equality in \eqref{eq:supsubset} by means of finding suitable subsolutions. This can be achieved in certain cases of interest. To proceed, we state first a general result that reduces this task to analyze what happens inside the support of a given subsolution ansatz.
\begin{lemma}
\label{lm:paso1}
Let $W(t,x)$ be such that $W(0,\cdot)$ is compactly supported. Assume that $B=B(t):= \mbox{supp}\, W(t,\cdot) = \mbox{supp}\, W(0,\cdot)\oplus tK$, with $K$ a compact, convex subset of $\R^d$ that is symmetric with respect to the origin and has nonempty interior. Assume further that $W$ satisfies the regularity requirements in Definition \ref{subsuper} (a sufficient condition is to assume smoothness inside the support and continuity at the boundary) and there holds that $W(t,\cdot)|_{\partial B}=0$. Then, provided that
$$
W_t \le \div \a(W,\nabla W)\quad \mbox{a.e. in} \ B(t), \quad \mbox{for a.e.}\ t \in (0,T)
$$
we have that $W$ satisfies \eqref{seineq} for equation \eqref{lm:coherence}.
\end{lemma}
\begin{proof} 
Let $0 \le \phi \in \mathcal{D}(Q_T)$, $T \in \mathcal{T}^+$ and $S\in \mathcal{T}^-$. We compute each term in \eqref{seineq} of Definition \ref{subsuper} separately. First, note that the spatial derivative has no singular part. Then we compute
\begin{equation}
\label{t2}
\partial_t J_{TS}(W) = \partial_tW \, T(W) S(W) \1_B. 
\end{equation}
Moreover, letting $\z = \a(W,\nabla W)$, 
\begin{equation}
\label{t3}
\begin{array}{ll}
\displaystyle \int_{Q_T} \z \nabla \phi T(W)S(W) \ dxdt = & \displaystyle -\int_{Q_T}  \phi\,  \div (\z T(W) S(W)) \ dxdt
\\ \\
& \displaystyle +\int_0^T \int_{\partial B} [\z T(W)S(W)\cdot \nu^B]\phi\,  d \H^{d-1}\ dt.
\end{array}
\end{equation}
We claim that the boundary term above vanishes; this is due to the estimate $|\z|\le |W| \, \|\psi\|_\infty$ -recall that $\psi$ is a $C^1$ vector field whose derivatives decay at infinity. Thus, \eqref{seineq} actually reads as follows:
\begin{equation}
\label{eq:med}
\begin{array}{l}
\displaystyle \int_{Q_T} \left[h_S(W,DT(W))^{ac} + h_T(W,DS(W))^{ac} \right] \phi \ dt
\\
\displaystyle \ge  \int_{Q_T}  \phi\,  \div (\z T(W) S(W)) \ dxdt - \int_0^T \int_B \phi \partial_tW\, T(W) S(W) \ dxdt.
 \end{array}
\end{equation}
We want to show that \eqref{eq:med} holds true under the conditions stated in the Lemma. To prove this, we compute
\[
\begin{split}
h_S(W,DT(W))^{ac} & = S(W) h(W,\nabla T(W)) 
\\
& = S(W) \nabla T(W) \a(T(W),\nabla T(W)) = S(W) \nabla T(W) \a(W,\nabla W)
\end{split}
\]
(and in the same way for the other term); then we have that
$$
h_S(W,DT(W))^{ac} + h_T(W,DS(W))^{ac} = \a(W,\nabla W) \nabla (S(W)T(W)).
$$
Thus, \eqref{eq:med} reduces to
$$
\int_{Q_T}  \phi T(W) S(W) \div \z \ dxdt - \int_{Q_T}  \phi \partial_t W\,  T(W) S(W) \ dxdt \le 0.
$$
This proves the Lemma.
\end{proof}

We are able to complete the previous program for the subclass of equations described by Lemma \ref{lm:example}:
\begin{proposition}
\label{prop:sub}
Let $R_0>0$. Under the assumptions of Lemma \ref{lm:example} and the additional assumption
\begin{equation}
\label{eq:extralimit}
\lim_{z\to +\infty} z^{3/2}g'(z)=-\alpha \in [-1/2,0)\:,
\end{equation}
there exists some $a>0$ (depending on $d,\alpha, \|g\|_\infty$ and $\|g'\|_\infty$) such that
$$
W(t,x)= e^{-at} \sqrt{R(t)^2-\|x\|_{A^{-1}}^2} \1_{\{R(t)\bar B_{A^{-1}}\}}\quad \mbox{with}\, R(t):= R_0+t
$$
 is a subsolution of \eqref{eq:anisotropic_template3}.
\end{proposition}
\begin{remark}
Condition \eqref{eq:extralimit} amounts to the additional assumption of Proposition 6.5 in \cite{Calvo2015} -note however that the formulation there is slightly different as $g$ is a function of the norm, whereas here $g$ is a function of the norm squared. We also point out that, thanks to \eqref{eq:newdecay} we have that
$$
|z^{3/2}g'(z)|\le \frac{1}{2} |\sqrt{z}g(z)| \quad \mbox{and hence} \quad \lim_{z\to +\infty} z^{3/2}|g'(z)|\le \frac{1}{2^{3/2}}.
$$
The equality in the former limit is fulfilled by every model of the form \eqref{eq:pmodel}. 
\end{remark}
\begin{proof}
Thanks to Lemma \ref{lm:paso1} it suffices to carry the computations within the support. We readily compute
$$
 \nabla W = -\frac{e^{-at} A^{-1}x^T}{\sqrt{R(t)^2-\|x\|_{A^{-1}}^2} },\quad \frac{\nabla W}{W} = -\frac{A^{-1}x^T}{R(t)^2-\|x\|_{A^{-1}}^2 }.
$$
Now, recall that $\|x\|_A^2=x A x^T$ for row vectors and likewise $\|x\|_A^2=x^TA x$ for column vectors. In this fashion, it is easy to see that 
$$
\|A^{-1}x^T \|_A= \|x\|_{A^{-1}}\quad \mbox{and hence}\quad \left\|\frac{\nabla W}{W}\right\|_A = \frac{\|x\|_{A^{-1}}}{R(t)^2-\|x\|_{A^{-1}}^2 }.
$$
Thanks to Lemma \ref{lm:paso1}, if we check \eqref{seineq} we are done. Under the structure given by \eqref{eq:Apsi}, we have to check whether 
$$
W_t \le \div \left( g\left(\frac{1}{2}\left\|\frac{\nabla W}{W} \right\|_A^2\right) A\nabla W\right)
$$
holds true a.e. within the support of $W$. Therefore, we are to check that
\begin{equation}
\label{eq:2page}
\begin{array}{r}
\displaystyle -a W + \frac{ e^{-at}R(t)}{\sqrt{R(t)^2-\|x\|_{A^{-1}}^2}} \le \div (A \cdot \nabla W) g\left(\frac{\|x\|_{A^{-1}}^2/2}{(R(t)^2-\|x\|_{A^{-1}}^2)^2} \right)
\\ \\
\displaystyle + (A \nabla W)\cdot \nabla \left(\left\|\frac{\nabla W}{W}\right\|_A \right) \left\|\frac{\nabla W}{W}\right\|_A g'\left(\frac{\|x\|_{A^{-1}}^2/2}{(R(t)^2-\|x\|_{A^{-1}}^2)^2} \right)
\end{array}
\end{equation}
holds inside the support of $W$. Let us compute the terms on the right hand side of \eqref{eq:2page} in turn. First,
$$
A \nabla W = - \frac{e^{-at}x^T}{\sqrt{R(t)^2-\|x\|_{A^{-1}}^2}}
$$
and in such a way
$$
\div (A \nabla W ) = \frac{e^{-at}\left((d+1)\|x\|_{A^{-1}}^2 - dR(t)^2 \right)}{(R(t)^2-\|x\|_{A^{-1}}^2)^{3/2}}
$$
where we recall that $d$ is the spatial dimension. Now,
$$
\nabla \left(\left\|\frac{\nabla W}{W}\right\|_A \right)  = \frac{A^{-1}x^T}{R(t)^2-\|x\|_{A^{-1}}^2} \left(\frac{1}{\|x\|_{A^{-1}}}+\frac{2\|x\|_{A^{-1}}}{R(t)^2-\|x\|_{A^{-1}}^2} \right)
$$
and therefore
$$
(A \nabla W) \cdot \nabla \left(\left\|\frac{\nabla W}{W}\right\|_A \right)  = -\frac{e^{-at}(\|x\|_{A^{-1}}[R(t)^2+2]-\|x\|_{A^{-1}}^3)}{[R(t)^2-\|x\|_{A^{-1}}^2]^{5/2}}.
$$
Putting all the contributions together and rearranging a bit we get the following, equi\-valent inequality:
$$
\begin{array}{r}
\displaystyle -a \le -\frac{\sqrt{2}R(t)}{R(t)^2-\|x\|_{A^{-1}}^2} + g\left(\frac{\|x\|_{A^{-1}}^2/2}{(R(t)^2-\|x\|_{A^{-1}}^2)^2} \right) \frac{(d+1)\|x \|_{A^{-1}}^2-d R(t)^2}{\left(R(t)^2-\|x\|_{A^{-1}}^2 \right)^2}
\\ \\
\displaystyle - g'\left(\frac{\|x\|_{A^{-1}}^2/2}{(R(t)^2-\|x\|_{A^{-1}}^2)^2} \right) \frac{\|x \|_{A^{-1}}^2[R(t)^2-\|x\|_{A^{-1}}^2]+2 \|x \|_{A^{-1}}^2}{\left(R(t)^2-\|x\|_{A^{-1}}^2 \right)^4}\:.
\end{array}
$$
The former inequality depends on $x\in \mbox{supp}\, W(t)$ only through $\|x\|_{A^{-1}}$. For each such $x$ we can find $\lambda \in [0,1)$ such that $\|x\|_{A^{-1}}^2=\lambda R(t)^2$. We can rewrite the above inequality in terms of $\lambda$:
\begin{equation}
\label{eq:no2page}
\begin{array}{r}
\displaystyle - a \le \frac{1}{(1-\lambda)R(t)} \left\{-\sqrt{2} + \frac{(d+1)\lambda-d}{R(t)(1-\lambda)}g\left(\frac{\lambda/2}{(1-\lambda)^2R(t)^2} \right) \right.
\\ \\
\displaystyle \left. -\frac{(2+R(t)^2)\lambda -\lambda^2R(t)^2}{(1-\lambda)^3R(t)^5}g'\left(\frac{\lambda/2}{(1-\lambda)^2R(t)^2} \right) \right\}\:.
\end{array}
\end{equation}
Since $a$ is finite but can be taken as large as needed, we just have to check that the limit of the right hand side of \eqref{eq:no2page} when $\lambda \to 1$ does not diverge to minus infinity. Note that
$$
\lim_{\lambda \to 1} \frac{(d+1)\lambda-d}{\sqrt{\lambda/2}} \frac{\sqrt{\lambda/2}}{R(t)(1-\lambda)}g\left(\frac{\lambda/2}{(1-\lambda)^2R(t)^2} \right)=\sqrt{2}
$$
thanks to \eqref{eq:newdecay}. 
Moreover, using \eqref{eq:extralimit} we can study the limit
$$
\lim_{\lambda \to 1} -\frac{(2+R(t)^2)\lambda -\lambda^2R(t)^2}{(\lambda/2)^{3/2}R(t)^2}
 \frac{(\lambda/2)^{3/2}}{(1-\lambda)^3R(t)^3}g'\left(\frac{\lambda/2}{(1-\lambda)^2R(t)^2} \right) =2^{5/2}\alpha>0. 
$$
This suffices to ensure that we can fulfill the inequality for $\lambda$ in a neighborhood of $\lambda =1$ just by taking $a$ to be large enough, which ends the proof.
\end{proof}

\begin{corollary}
\label{cor:47}
 Let $0\le u_0\in L^\infty(\R^d)$ be a compactly supported initial datum whose support is the closure of its interior. Let $\mbox{supp}\, u_0$ be either convex or such that its boundary is of class $C^2$. Assume that
 \begin{equation}
 \label{ass:F}
 \begin{array}{l}
 \mbox{for any closed ball}\, F\, \mbox{of the norm}\, \|\cdot \|_{A^{-1}}\, \mbox{which is contained in}\, {\rm int}( \mbox{supp}\, u_0),
 \\
  \mbox{there is a constant}\,  \alpha_F>0\, \mbox{such that}\, \alpha_F \le u_0\, \mbox{in}\, F.
  \end{array}
 \end{equation}
 Let $u(t)$ be the entropy solution of \eqref{eq:anisotropic_template3} with initial datum $u_0$. Then there holds that $\mbox{supp}\, u(t) = \mbox{supp}\, u_0 \oplus t\bar B_{A^{-1}}$.
\end{corollary}
\begin{proof}
This follows by comparison using Propositions \ref{prop:super} and \ref{prop:sub}. Note that any positive multiple of the subsolutions in Proposition \ref{prop:sub} is again a subsolution. We use condition \eqref{ass:F} in order to be able to place these subsolutions below $u_0$ and centered about points $x\in {\rm int} ({supp}\, u_0)$ arbitrarily close to the boundary. This ensures that  we can find some $\epsilon>0$ such that $x+\epsilon B_{A^{-1}} \subset {supp}\, u_0$. 
\end{proof}

\section{Propagation of jump discontinuities}
\label{sec:RH}
The aim of this section is to describe the time evolution of jump discontinuities, generalizing what was already shown in \cite{Caselles2011,Calvo2015}. We show that jump hypersurfaces described by a normal vector $\nu$ propagate along that vector with speed $\Phi_\infty(\nu)$, see Proposition \ref{pr:8.1} below.

To proceed we will follow closely the framework introduced by \cite{Caselles2011}. Note that this functional framework demands $u \in BV(Q_T)$; to comply with this requirement we must have that $u_t$ is a Radon measure, which is a delicate issue from the technical point of view (see e.g. the discussion in \cite{Caselles2011}, sufficient conditions are given for the case of \eqref{eq:rhe0}). This falls out of the scope of the present document; here we will just take for granted that we work with solutions having the appropriate time regularity. 

Let us show that the entropy inequalities \eqref{eineq} can be decomposed in terms of jump and Cantor parts, in the same philosophy of \cite{Caselles2011}.
\begin{proposition}
\label{pr:7.1}
Let $u \in C([0,T];L^1(\R^d))\cap BV_{loc}(Q_T)$. Assume that $u_t=\div \z$ in $\mathcal{D}'(Q_T)$, where $\z = \a(u,\nabla u)$. Assume also that $u_t(t)$ is a Radon measure for a.e. $t>0$. Let $I:\R \rightarrow \R$ be the identity function $I(z)=z$. Then $u$ is an entropy solution of \eqref{eq:anisotropic_template} if and only if for any $(T,S) \in \mathcal{TSUB}$ (for any $(T,S) \in \mathcal{TSUB}\cup \mathcal{TSUPER}$) we have
$$
h_S(u,DT(u))^c+h_T(u,DS(u))^c \le (\z(t,x)\cdot D(T(u)S(u)))^c
$$
and for almost any $t>0$ the inequality
\begin{equation}
\begin{array}{c}
\label{entcondsimple}
\displaystyle \Phi_\infty(\nu^{J_{u(t)}})\{[(STI)(u(t))]_{+-}  \displaystyle- \ [J_{TS}(u(t))]_{+-} \}
\\ \\
\quad \quad \quad \quad \quad \quad  \displaystyle \le - \v [J_{TS}(u(t))]_{+-} + [[\z(t)\cdot \nu^{J_{u(t)}} ]T(u(t))S(u(t))]_{+-}
\end{array}
\end{equation}
holds $\mathcal{H}^{d-1}$-a.e. on $J_{u(t)}$.
\end{proposition}
\begin{proof}
This is essentially the same proof as that of Proposition 7.1 in \cite{Caselles2011}. The only change needed pertains the computation of the jump parts for the measures $h_S, \, h_T$. This is similar to what we did in the proof of Proposition \ref{prop:super}. Actually, we have
$$
h_S(u(t),DT(u(t)))^j =  \Phi_\infty \left(\frac{DT(u(t))}{|DT(u(t))|}\right) [J_{TIS'}(u(t))]_{+-} \H_{J_{u(t)}}^{d-1}\, dt
$$
for any $(T,S) \in \mathcal{TSUB}$
 and similarly for $h_T(u(t),DS(u(t)))^j$. Since this measure is concentrated at $J_{u(t)}$, we have that $DT(u(t))=D^jT(u(t))$ on this set. Then we can argue as in the proof of Proposition \ref{prop:super} to deduce that $\frac{DT(u(t))}{|DT(u(t))|}=\nu^{J_{u(t)}}$ at $J_{u(t)}$. In this fashion we obtain that
$$
h_S(u(t),DT(u(t)))^j =  \Phi_\infty \left(\nu^{J_{u(t)}}\right) [J_{TIS'}(u(t))]_{+-} \H_{J_{u(t)}}^{d-1}\, dt
$$
and
$$
h_T(u(t),DS(u(t)))^j =  \Phi_\infty \left(\nu^{J_{u(t)}}\right) [J_{SIT'}(u(t))]_{+-} \H_{J_{u(t)}}^{d-1}\, dt\:.
$$
This is enough to readapt the arguments in Proposition 7.1, \cite{Caselles2011}.
\end{proof}
The central result of this section characterizes the structure of the fluxes across discontinuity points and provides the Rankine--Hugoniot formula for the speed of propagating jump discontinuities, in the line of the results in \cite{Caselles2011}.
\begin{proposition}
\label{pr:8.1}
Let $u \in C([0,T];L^1(\R^d))$ be the entropy solution of \eqref{eq:anisotropic_template} with $0\le u(0)=u_0 \in L^\infty(\R^d)\cap BV(\R^d)$. Assume that $u \in BV_{loc}(Q_T)$.  Then the entropy conditions \eqref{entcondsimple} hold if and only if for almost any $t \in (0,T)$
\begin{equation}
\label{entcondsimple2}
[\z \cdot \nu^{J_{u(t)}}]_+ = \Phi_\infty \left(\nu^{J_{u(t)}}\right) u^+(t) \quad \mbox{and} \quad [\z \cdot \nu^{J_{u(t)}}]_- = \Phi_\infty \left(\nu^{J_{u(t)}}\right) u^-(t)
\end{equation}
hold $\mathcal{H}^{d-1}$-a.e. on $J_{u(t)}$. Moreover the speed of any discontinuity front is 
\begin{equation}
\label{vformula}
\v=\Phi_\infty \left(\nu^{J_{u(t)}}\right)\:.
\end{equation}
\end{proposition}
\begin{proof}
The proof is a suitable generalization of that given for Proposition 8.1 in \cite{Caselles2011} -see also \cite{Calvo2015}. Let us show that \eqref{entcondsimple} implies \eqref{entcondsimple2}. For that we let $\epsilon >0$ be such that $u^- < u^+-\epsilon <u^+$ and we choose $(S,T)\in \mathcal{TSUB}$ so that $S(r)T(r)=(r-(u^+-\epsilon))^+$. Then we compute:
\begin{enumerate}
\item $[STI(u(t))]_{+-}= \epsilon u^+$,
\item $[J_{TS}(u(t))]_{+-}= \frac{\epsilon^2}{2}$,
\item $[[\z(t) \cdot \nu^{J_{u(t)}}]T(u(t))S(u(t))]_{+-}=\epsilon [\z(t) \cdot \nu^{J_{u(t)}}]_{+}$.
\end{enumerate}
Then \eqref{entcondsimple} is written as
\begin{equation}
\label{eq:e12}
\epsilon (\Phi_\infty (\nu^{J_{u(t)}}) u^+ -  [\z(t) \cdot \nu^{J_{u(t)}}]_{+}) \le\frac{\epsilon^2}{2}\left(\Phi_\infty (\nu^{J_{u(t)}}) -\v \right)\:.
\end{equation}
Here we draw attention to the fact that 
$$
| [\z(t) \cdot \nu^{J_{u(t)}}]_{+}|\le \Phi_\infty (\nu^{J_{u(t)}}) u^+
$$
-note that $D^j u(t)$ is aligned with $\nu^{J_{u(t)}}$. Hence \eqref{eq:e12} leads to a contradiction, unless $ [\z(t) \cdot \nu^{J_{u(t)}}]_{+}= \Phi_\infty (\nu^{J_{u(t)}})  u^+$. We show that $ [\z(t) \cdot \nu^{J_{u(t)}}]_{-}= \Phi_\infty (\nu^{J_{u(t)}}) u^-$ in a similar way. Therefore \eqref{entcondsimple2} holds -and \eqref{vformula} follows, see below.

Let us show now the converse implication. Thanks to \eqref{entcondsimple2} we may write
\[
\begin{split}
[[\z \cdot \nu^{J_{u(t)}}]T(u)S(u)]_{+-} & = [\z \cdot \nu^{J_{u(t)}}]_+T(u^+)S(u^+) - [\z \cdot \nu^{J_{u(t)}}]_-T(u^-)S(u^-)
\\
& = \Phi_\infty \left(\nu^{J_{u(t)}}\right) u^+T(u^+)S(u^+) - \Phi_\infty \left(\nu^{J_{u(t)}}\right) u^-T(u^-)S(u^-)
\\
& = \Phi_\infty \left(\nu^{J_{u(t)}}\right) [STI(u(t))]_{+-}.
\end{split}
\]
Thus, we recast \eqref{entcondsimple} as
\begin{equation}
\label{entcondsimple3}
\v [J_{TS}(u(t))]_{+-} \le \Phi_\infty \left(\nu^{J_{u(t)}}\right) [J_{TS}(u(t))]_{+-}\:.
\end{equation}
Recall that the Rankine--Hugoniot conditions stated in Lemma \ref{lem:conrh} are 
$$
 \v [u]_{+-} = [[\z \cdot \nu^{J_{u(t)}}]]_{+-}.
$$
Owing to \eqref{entcondsimple2} we have that 
$$
\v = \frac{[\z \cdot \nu^{J_{u(t)}}]_{+}- [\z \cdot \nu^{J_{u(t)}}]_{-}}{u^+- u^-} = \Phi_\infty \left(\nu^{J_{u(t)}}\right).
$$
Thus \eqref{entcondsimple3} is trivially satisfied. This proves that \eqref{entcondsimple2} implies the fulfillment of the entropy conditions.
\end{proof}
\begin{remark}
We note that for solutions having no Cantor part we have equality in the entropy inequality \eqref{entcondsimple}. This is not necessarily the case in more general situations, like the porous media variants discussed in \cite{Caselles2011}.
\end{remark}

\section{Appendix: a primer on entropy solutions}
\label{sec:appendix}
We briefly recall here the functional framework introduced in \cite{Andreu2005elliptic,Andreu2005parabolic} to deal with equations of the form 
\begin{equation}
\label{eq:generaltemplate}
u_t=\div {\bf a}\, (u,\nabla u)
\end{equation}
 such that the Lagrangian of ${\bf a}\, (z,\xi)$ grows linearly for $|\xi|\to \infty$. We follow closely the presentation in \cite{Carrillo2013} -see also \cite{Calvo2015survey} for more details.

\subsection{Several classes of truncation functions}
We will use in the sequel a number of different truncation functions. For $a <
b$ and $l \in \R$, let $T_{a,b}(r) := \max \{\min \{b,r\},a\}$, $ T_{a,b}^l =  T_{a,b}-l$.
We denote \cite{Andreu2005elliptic,Andreu2005parabolic}
\[
\begin{split}
\mathcal T_r & := \{ T_{a,b} \ : \ 0 < a < b \},  
\\
\mathcal{T}^+ & := \{ T_{a,b}^l \ : \ 0 < a < b ,\, l\in \R, \, T_{a,b}^l\geq 0 \},  
\\
\mathcal{T}^- & := \{ T_{a,b}^l \ : \ 0 < a < b ,\, l\in \R, \, T_{a,b}^l\leq 0 \},
\\
{\mathcal P} & := \{ p : [0, +\infty) \rightarrow \R\, \mbox{Lipschitz},\, p'(s)=0 \, \mbox{for large enough}\, s \},
\\
{\mathcal P}^+ & := \{ p \in {\mathcal P} \ : \ p \geq 0 \}.
\end{split}
\]

We need to consider the following function space
$$
TBV_{\rm r}^+(\R^d):= \left\{ w \in L^1(\R^d)^+  \ :  \ \ T_{a,b}(w) - a \in BV(\R^d), \
\ \forall \ T_{a,b} \in \mathcal T_r \right\}.
$$
\begin{remark}
Using the chain rule for BV-functions (see for instance
\cite{Ambrosio2000}), one can give a sense to $\nabla u$ for a
function $u \in TBV^+(\R^d)$ as the unique function $v$ which
satisfies
\begin{equation*}
\label{E1WRN}
\nabla T_{a,b}(u) = v \1_{\{a < u  < b\}} \ \ \ \ \ {\mathcal
L}^d-a.e., \ \ \forall \ T_{a,b} \in \mathcal{T}_r.
\end{equation*}
We refer to \cite{Benilan1995} for details.
\end{remark}

An extended class of truncation functions was introduced by Caselles in order to assess fine properties of entropy solutions. Following \cite{Caselles2011,Caselles2011bis}, we introduce  $\mathcal{TSUB}$ as the class of functions $S,T \in \mathcal{P}$ such that 
$$
S\ge 0, \quad S'\ge 0 \quad \mbox{and} \quad T\ge 0, \quad T'\ge 0
$$
and $p(r)=\tilde{p}(T_{a,b}(r))$ for some $0<a<b$, being $\tilde{p}$ differentiable in a neighborhood of $[a,b]$ and $p=S,T$. Similarly, let $\mathcal{TSUPER}$ be the class of functions $S,T \in \mathcal{P}$ such that 
$$
S\le 0, \quad S'\ge 0 \quad \mbox{and} \quad T\ge 0, \quad T'\le 0
$$
and $p(r)=\tilde{p}(T_{a,b}(r))$ for some $0<a<b$, being $\tilde{p}$ differentiable in a neighborhood of $[a,b]$ and $p=S,T$.

\subsection{Functional calculus}\label{sect:functionalcalculus}

In order to define the notion of entropy solutions of
(\ref{eq:generaltemplate}) we
need a functional calculus defined on functions whose truncations
are in $BV$. For that we need to introduce some functionals defined on functions of bounded variation \cite{Andreu2005elliptic,Andreu2005parabolic}. Let $\Omega$ be an open subset of $\R^d$ and consider $g: \Omega \times \R
\times \R^d \rightarrow [0, \infty)$ a locally bounded Caratheodory function.
 After Dal Maso \cite{DalMaso1980} we consider the following  
functional:
\begin{equation}
\nonumber
\begin{array}{ll}
\displaystyle {\mathcal R}_g(u) := \int_{\Omega} g(x,u(x), \nabla
u(x)) \, dx + \int_{\Omega} g_\infty \left(x,
\tilde{u}(x),\frac{Du}{|Du|}(x) \right) \,  d\vert D^c
u \vert \\ \\
\qquad \qquad \displaystyle + \int_{J_u} \left(\int_{u_-(x)}^{u_+(x)}
g_\infty \left(x, s, \frac{Du}{|Du|}(x) \right) \, ds \right)\, d \H^{d-1}(x),
\end{array}
\end{equation}
for $u \in BV(\Omega) \cap L^\infty(\Omega)$, being $\tilde{u}$
the approximated limit of $u$ (see e.g. \cite{Ambrosio2000}). The recession
function $g_\infty$ of $g$ with respect to its third variable is defined by
\begin{equation}
\label{Asimptfunct}
 g_\infty(x, z, \xi) = \lim_{t \to 0^+} t\, g \left(x, z, \xi/t
 \right).
\end{equation}
It is convex and homogeneous of degree $1$ in $\xi$. 

The following semi-continuity result follows from \cite{DeCicco2005}:
\begin{theorem}
\label{th:DC}
Let $g$ above satisfy the following properties:
\begin{enumerate}
\item For every $(z,\xi)\in \R \times \R^d$, the function $g(\cdot,z,\xi)$ is of class $C^1$.

\item For every $(x,z)\in \Omega \times \R$, the function $g(x,z,\cdot)$ is convex.

\item For every $(x,\xi)\in \Omega \times \R^d$, the function $g(x,\cdot,\xi)$ is continuous.
\end{enumerate}
Then ${\mathcal R}_g(u)$ is lower semicontinuous with respect to the $L^1(\Omega)$-convergence.
\end{theorem}

Let us consider now the following functional,  defined in $TBV^+(\R^d)$:
\[
\mathcal{R}(g,T)(u):=\mathcal{R}_g(T_{a,b}(u))+ \int_{[u\le a]} (g(x,u(x),0)-g(x,a,0))\, dx
\]
\[
+ \int_{[u\ge b]} (g(x,u(x),0)-g(x,b,0))\, dx\:.
\]

\begin{definition}[measures $g(u, DT(u))$]
Assume now that $g:\R\times \R^d \to [0, \infty)$ is a Borel function
such that
\begin{equation}
\nonumber
C \vert \xi \vert - D \leq g(z, \xi)  \leq M(1+ \vert \xi \vert)
\qquad \forall (z,\xi)\in \R^d, \, \vert z \vert \leq R,
\end{equation}
for any $R > 0$ and for some constants  $C,D,M \geq 0$ which may
depend on $R$. Assume also that, for any $u\in L^1(\R^d)^+$
\begin{equation}
\nonumber
\1_{\{u\leq a\}} \left(g(u(x), 0) - g(a, 0)\right), \1_{\{u \geq b\}} \left(g(u(x),
0) - g(b, 0) \right) \in L^1(\R^d).
\end{equation}
Let $u \in TBV_{\rm r}^+(\R^d)  \cap
L^\infty(\R^d)$  and $T = T_{a,b}-l\in {\mathcal T}^+ $. For each
$\phi\in \mathcal{C}_c(\R^d)$, $\phi \geq 0$, we define the Radon measure
$g(u, DT(u))$ by
\begin{equation}
\label{FUTab}
\begin{array}{c}
\displaystyle \langle g(u, DT(u)), \phi \rangle : = {\mathcal R}(\phi g,T)(u)= {\mathcal R}_{\phi g}(T_{a,b}(u))+
\displaystyle\int_{\{u \leq a\}} \phi(x)
\left( g(u(x), 0) - g(a, 0)\right) \, dx 
\\
 \displaystyle
\displaystyle  + \int_{\{u \geq b\}} \phi(x)
\left(g(u(x), 0) - g(b, 0) \right) \, dx.
\end{array}
\end{equation}
If $\phi\in \mathcal{C}_c(\R^d)$, we write $\phi = \phi^+ -
\phi^-$ 
and we define 
$$\langle g(u, DT(u)), \phi \rangle : =
\langle g(u, DT(u)), \phi^+ \rangle- \langle g(u, DT(u)), \phi^- \rangle.$$
\end{definition}

Note that, as a consequence of Theorem \ref{th:DC}, the following holds: if $g(z,\xi)$ is continuous in $(z,\xi)$, convex
in $\xi$ for any $z\in \R$, and $\phi \in \mathcal{C}^1(\R^d)^+$ has compact
support, then  $\langle g(u, DT(u)), \phi \rangle$ is lower
semi-continuous in $TBV^+(\R^d)$ with respect to the 
$L^1(\R^d)$-convergence.

\begin{definition}[measures $g_S(u, DT(u))$]
 Let $S \in  \mathcal{P}^+$, $T \in \mathcal{T}^+$.
We assume that
$u \in TBV_{\rm r}^+(\R^d) \cap L^\infty(\R^d)$ and 
$$
\1_{\{u\leq a\}} S(u)\left(g(u(x), 0) - g(a, 0)\right), \1_{\{u \geq b\}} S(u)\left(g(u(x),
0) - g(b,0) \right) \in L^1(\R^d).
$$
Then we define $g_S(u,DT(u))$ as the Radon
measure given by (\ref{FUTab}) with $g_{S}(z,\xi) = S(z)
g(z,\xi)$. 
\end{definition}

Let us introduce $f,h: \R \times \R^d \rightarrow \R$ defined by
\begin{equation}
\label{hdef}
\a(z,\xi)=\nabla_\xi f(z,\xi)\quad \mbox{and}\quad h(z,\xi):=\a(z,\xi) \xi,
\end{equation}
being $\a$ the flux in \eqref{eq:generaltemplate}. 

\begin{definition}[measures generated by $f$ and $h$]
We introduce the measures $h(u,DT(u))$, $h_S(u,DT(u))$ as those generated by $g(z,\xi) =h(z,\xi) $ and $g_S(z,\xi)=S(z) h(z,\xi)$. In a similar way, the measures $f(u,DT(u)), f_S(u,DT(u))$ are those generated by $g(z,\xi) =f(z,\xi) $ and $g_S(z,\xi)=S(z) f(z,\xi)$.
\end{definition}


\subsection{Entropy solutions of the evolution problem}\label{sect:defESpp}

Let $L^1_{w}(0,T,BV(\R^d))$  be the space of weakly$^*$
measurable functions $w:[0,T] \to BV(\R^d)$ (i.e., $t \in [0,T]
\to \langle w(t),\phi \rangle$ is measurable for every $\phi$ in the predual
of $BV(\R^d)$) such that $\int_0^T \| w(t)\|_{BV} \, dt$ is finite.
Observe that, since $BV(\R^d)$ has a separable predual (see
\cite{Ambrosio2000}), it follows easily that the map $t \in [0,T]\to
\Vert w(t) \Vert_{BV}$ is measurable. By  $L^1_{loc, w}(0, T,
BV(\R^d))$ we denote the space of weakly$^*$ measurable functions
$w:[0,T] \to BV(\R^d)$ such that the map $t \in [0,T]\to \Vert
w(t) \Vert_{BV}$ is in $L^1_{loc}(0, T)$.

We let $J_q(r)$ denote the primitive of $q$ for any real function $q$; i.e. 
$$\displaystyle J_q(r):=\int_0^r q(s)\,ds.$$

\begin{definition} 
\label{def:espb}
Assume that $0 \le u_0 \in L^1(\R^d)\cap L^\infty(\R^d)$. A
measurable function $u: (0,T)\times \R^d \rightarrow \R$ is an
{\it entropy  solution} of \eqref{eq:generaltemplate} in $Q_T:=(0,T)\times \R^d$ with initial datum $u_0$ if $u \in \mathcal{C}([0, T]; L^1(\R^d))$,
$T_{a,b}(u(\cdot)) - a \in L^1_{loc, w}(0, T, BV(\R^d))$ for all
$0 < a < b$, and
\begin{itemize}
\item[(i)] \ $u_t = {\rm div} \, \a(u(t), \nabla u(t))$ in $\mathcal{D}^\prime(Q_T)$,
\item[(ii)]  $u(0) = u_0$, and 
\item[(iii)] \ the following
inequality is satisfied
\begin{equation}
\begin{array}{c}
 \label{eineq}
\displaystyle \int_0^T\int_{\R^d} \phi
h_{S}(u,DT(u)) \, dt + \int_0^T\int_{\R^d} \phi h_{T}(u,DS(u)) \, dt
\\
\leq  \displaystyle\int_0^T\!\int_{\R^d}\! \Big\{ \!J_{TS}(u(t)) \phi^{\prime}(t) - \a(u(t), \nabla u(t))\! \cdot \!\nabla \phi 
T(u(t)) S(u(t))\!\Big\} dxdt, 
\end{array}
\end{equation}
 for truncation functions $S,  T \in \mathcal{T}^+$, and any  smooth function $\phi$ of
 compact support, in particular  those  of the form $\phi(t,x) =
 \phi_1(t)\rho(x)$, $\phi_1\in {\mathcal D}(0,T)$, $\rho \in
 {\mathcal D}(\R^d)$.
\end{itemize}
\end{definition}
\noindent
{This definition is a simplification of the original one in \cite{Andreu2005parabolic}, see \cite{Andreu2008} for instance}; a related notion using the class $\mathcal{TSUB}$ can be found in \cite{Caselles2011bis}. Note that the statements in this paragraph and the following one hold under a set of assumptions on $\a$ that are described in \cite{Andreu2005elliptic,Andreu2007,Caselles2011}, which we denote collectively by $({\rm H})$. We have the following existence and uniqueness result \cite{Andreu2005parabolic}.

\begin{theorem}
\label{EUTEparabolic}
 Let assumptions
$({\rm H})$ hold. Then, for any initial datum $0 \leq u_0 \in L^1(\R^d)\cap L^\infty(\R^d)$
there exists a unique entropy solution $u$ of
\eqref{eq:generaltemplate} in $Q_T $ for every $T
> 0$ such that $u(0) = u_0$.
Moreover, if $u(t)$,
$\overline{u}(t)$ are the entropy solutions corresponding to
initial data $u_0$, $\overline{u}_0 \in L^{1}(\R^d)^+$ respectively, then
\begin{equation}
\nonumber
\Vert (u(t) - \overline{u}(t))^+ \Vert_1 \leq
\Vert (u_0 - \overline{u}_0)^+ \Vert_1 \ \ \ \ \ \ {\rm for \
all} \ \ t \geq 0,
\end{equation}
where we define the positive part as $u^+(t,x) = \max \{u(t,x),0\}$.
\end{theorem}

Existence of entropy solutions is proved by using Crandall--Liggett's scheme \cite{Crandall1971}
and uniqueness is proved using Kruzhkov's doubling variables technique \cite{Kruzhkov1970,Carrillo1999}.

\subsection{Sub- and supersolutions}
In order to use the comparison principles introduced in \cite{Andreu2006} a certain technical condition is required.
\begin{assumptions}
\label{compas}
Let the function $h$ {defined by} \eqref{hdef} satisfy
$
h(z,\xi) \le M(z) |\xi|
$
for some positive continuous function $M(z)$ and for any $(z,\xi)\in \R \times \R^d$.
\end{assumptions}
This condition is satisfied in our framework since $|h(z,\xi)|\le \|\psi\|_\infty |z|\, |\xi|$, check Definition \ref{df:list}.

\begin{definition}{\rm \cite{Andreu2006}}
\label{subsuper}
A measurable function $u : (0,T) \times \R^d \rightarrow \R_0^+$ is an entropy sub- (resp. super-) solution of \eqref{eq:generaltemplate} if $u \in C([0,T],L^1(\R^d))$, $\a(u,\nabla u)\in L^\infty(Q_T)$, $T_{a,b}(u) \in L_{loc,w}^1(0,T,BV(\R^d))$ for every $0<a<b$ and the following inequality is satisfied:
\begin{equation}
\begin{array}{c}
\displaystyle \int_0^T\int_{\R^d} \phi
h_{S}(u,DT(u)) \, dt + \int_0^T\int_{\R^d} \phi h_{T}(u,DS(u)) \, dt
\\
\geq  \displaystyle\int_0^T\int_{\R^d} \Big\{ J_{TS}(u(t)) \phi^{\prime}(t) - \a(u(t), \nabla u(t)) \cdot \nabla \phi \
T(u(t)) S(u(t))\Big\} dxdt,  
\label{seineq}
\end{array}
\end{equation}
(whereas for supersolutions we require
\begin{equation}
\begin{array}{c}
\displaystyle \int_0^T\int_{\R^d} \phi
h_{S}(u,DT(u)) \, dt + \int_0^T\int_{\R^d} \phi h_{T}(u,DS(u)) \, dt
\\
\leq  \displaystyle\int_0^T\int_{\R^d} \Big\{ J_{TS}(u(t)) \phi^{\prime}(t) - \a(u(t), \nabla u(t)) \cdot \nabla \phi \
T(u(t)) S(u(t))\Big\} dxdt  \, )
\label{seineqbis}
\end{array}
\end{equation}
for any $\phi \in \mathcal{D}(Q_T)^+$ and any truncations $T \in \mathcal{T}^+,\ S \in \mathcal{T}^-$.
\end{definition}
This implies that 
\begin{equation}
\label{sinside}
u_t \le \div \a(u,\nabla u) \quad \mbox{in}\ \mathcal{D}'(Q_T)
\end{equation}
(resp. with $\ge$). The following comparison principle was shown in \cite{Andreu2006}:
\begin{theorem}
\label{theocomp}
Let assumptions (H) and Assumptions {\rm \ref{compas}} hold. Given an entropy solution $u$ of \eqref{eq:generaltemplate} corresponding to an initial datum $0\le u_0\in (L^\infty \cap L^1)(\R^d)$, the following statements hold true:
\begin{enumerate}
\item if $\overline{u}$ is a supersolution of \eqref{eq:generaltemplate} such that $\overline{u}(t) \in BV(\R^d)$ for a.e. $t \in (0,T)$, then
$$
\|(u(t)-\overline{u}(t))^+\|_1 \le \|(u_0-\overline{u}(0))^+\|_1 \quad \forall t \in [0,T],
$$
\item if $\overline{u}$ is a subsolution of \eqref{eq:generaltemplate} such that $\overline{u}(t) \in BV(\R^d)$ for a.e. $t \in (0,T)$, then
$$
\|(\overline{u}(t)-u(t))^+\|_1 \le \|(\overline{u}(0)-u_0)^+\|_1 \quad \forall t \in [0,T].
$$
\end{enumerate}
\end{theorem}
Although we will not require them, some extensions of this result have been shown in \cite{Giacomelli2015}, which enable to consider some instances of sub- and supersolutions that are not globally integrable.

\subsection{Rankine--Hugoniot jump conditions}
We borrow some notations from \cite{Caselles2011}. Assume that $u \in BV_{loc}(Q_T)$. Let $\nu:=\nu_u=(\nu_t,\nu_x)$ be the unit normal to the jump set of $u$ and $\nu^{J_{u(t)}}$ the unit normal to the jump set of $u(t)$. We write $[u](t,x):=u^+(t,x)-u^-(t,x)$ for the jump of $u$ at $(t,x)\in J_u$ and $[u(t)](x):=u(t)^+(x)-u(t)^-(x)$ for the jump of $u(t)$ at the point $x \in J_{u(t)}$. Without losing generality, we assume that $u^+>u^-$ in what follows; we also assume $u^-\ge 0$. 
\begin{definition}
\label{def:vsalto}
Let $u \in BV_{loc}(Q_T)$ and let $\z \in L^\infty ([0,T]\times \R^d,\R^d)$ be such that $u_t = \div \z$ in $\mathcal{D}'(Q_T)$. We define the speed of the discontinuity set of $u$ as $\v(t,x) = \frac{\nu_t(t,x)}{|\nu_x(t,x)|}\, \mathcal{H}^d$-a.e. on $J_u$.
\end{definition}
The following result summarizes some statements proved in \cite{Caselles2011}.
\begin{lemma}
\label{lem:conrh}
Let $u\in BV_{loc}(Q_T)$ and let $\z \in L^\infty([0,T]\times \R^d;\R^d)$ be such that $u_t = \div\, \z$ in $\mathcal{D}'(Q_T)$. Then:
\begin{enumerate}
\item There holds that
$$
\mathcal{H}^d \left(\{(t,x) \in J_u/ \nu_x(t,x)=0\}\right)=0
$$
and hence Definition \ref{def:vsalto} makes sense.

\item For a.e. $t\in (0,T)$ we have
$$
[u(t)](x) \v(t,x) = [[\z \cdot \nu^{J_{u(t)}}]]_{+-}\quad \mathcal{H}^{d-1}-\mbox{a.e. in}\ J_{u(t)},
$$
where $[[\z \cdot \nu^{J_{u(t)}}]]_{+-}$ denotes the difference of traces from both sides of $J_{u(t)}$.
\end{enumerate}
\end{lemma}

\medskip \noindent {\it Acknowledgements.}  J.C. acknowledges support from ``Plan Propio de Investigaci\'on, programa 9'' (funded by Universidad de Granada and european FEDER (ERDF) funds),  Project RTI2018-098850-B-I00 (funded by MICINN and european FEDER funds) and Projects A-FQM-311-UGR18 and P18-RT-2422 (funded by Junta de Andaluc\'ia and european FEDER funds). We thank Jos\'e M. Maz\'on and Juan Soler for their useful comments on an earlier version of this document.


\end{document}